%% file: main.tex
\documentclass[final,hidelinks,onefignum,onetabnum]{siamart251104}

\input{shared}

\ifpdf
\hypersetup{
  pdftitle={Structure-preserving integration for magnetic Gaussian wave packet dynamics},
  pdfauthor={Sebastian Merk, Caroline Lasser}
}
\fi

\begin{document}

\maketitle

% REQUIRED
\begin{abstract}
We develop structure-preserving time integration schemes for Gaussian wave packet dynamics associated with the magnetic Schr\"odinger equation. The variational Dirac--Frenkel formulation yields a finite-dimensional Hamiltonian system for the wave packet parameters, where the presence of a magnetic vector potential leads to a non-separable structure and a modified symplectic geometry. 
By introducing kinetic momenta through a minimal substitution, we reformulate the averaged dynamics as a Poisson system that closely parallels the classical equations of charged particle motion. This representation enables the construction of Boris-type integrators adapted to the variational setting. In addition, we propose explicit high-order symplectic schemes based on splitting methods and partitioned Runge--Kutta integrators. The proposed methods conserve the quadratic invariants characterizing the Hagedorn parametrization, preserve linear and angular momentum under symmetry assumptions, and exhibit near-conservation of the averaged Hamiltonian over long time intervals. Rigorous error estimates are derived for both the wave packet parameters and observable quantities, with bounds uniform in the semiclassical parameter. Numerical experiments demonstrate the favorable long-time behavior and structure preservation of the integrators.
\end{abstract}

% REQUIRED

% REQUIRED

\section{Introduction}
The time-dependent Schr\"odinger equation with electromagnetic potential arises in a wide range of applications, including molecular dynamics, charged particle transport, and plasma physics. Here, we consider the evolution in semiclassical scaling,
\begin{equation}\label{eq:Schro}
i\varepsilon\partial_t\psi(t) = \mathcal H(t)\psi(t),\quad \psi(t_0) = \psi_0,    
\end{equation}
with Hamiltonian operator 
\begin{equation}\label{eq:mag_ham}
\mathcal H(t) = \tfrac{1}{2}\abs{-i\varepsilon\nabla_x - A(t,x)}^2 + V(t,x),\quad x\in\mathbb R^d,
\end{equation}
and semiclassical parameter $\varepsilon>0$. 
In the presence of a magnetic field, the Hamiltonian involves a vector potential $A(t,\cdot)$ and a scalar potential $V(t,\cdot)$, which leads to non-separable dynamics with additional geometric structure. Accurate long-time simulation of such systems requires numerical methods that respect fundamental invariants such as symplecticity, linear and angular momentum conservation, and near-conservation of energy.

In the high-dimensional and/or highly-oscillatory regime, Gaussian wave packet methods provide an efficient reduced description of the dynamics by approximating the solution through parametrized families of localized states \cite{Heller_1976,CoalsonKarplus,KO2019}. 
Variational formulations based on the Dirac--Frenkel principle lead to closed systems of ordinary differential equations that retain essential geometric properties of the underlying Schr\"odinger dynamics.
If the wave packet parameters are chosen in Hagedorn's parametrization, then the approximate evolution has canonical Hamiltonian structure, which forms the basis of many successful semiclassical approximations \cite{LO2013,O2015,LL2020}. 

For nonmagnetic quantum dynamics ($A=0$), structure-preserving time integrators for Gaussian wave packet dynamics have been proposed in \cite{FL2006,FereidaniVanicek2023}. By preservation of the wave packet induced Poisson structure, linear and angular momenta are conserved, and energy has a favorable long-time behavior. In contrast, the presence of magnetic fields ($A\neq0$) introduces additional coupling through the vector potential, which leads to non-separable Hamiltonians and modified symplectic properties.
While variational Gaussian approximations for the magnetic Schr\"odinger equation have recently been derived and analyzed in the continuous-time setting \cite{KO2019,BDHL2023}, the Boris-type integrator proposed in \cite{SBHL2025} appears to be the only structure-adapted time discretization for magnetic variational Gaussian dynamics. A rigorous error analysis of this method, however, has not yet been established.

The present work develops a framework for structure-preserving time integration schemes for Gaussian wave packet dynamics in the presence of magnetic fields. By expressing the variational equations of motion in canonical coordinates, we construct 
averaged potentials $\AB(t,\qB)$ and $\VB(t,\qB)$ and then a parameter Hamiltonian
\[
\hB (t,\qB ,\pB )=\tfrac{1}{2}\pB ^\top \pB -\pB ^\top \AB (t,\qB )+\VB (t,\qB ),
\]
that closely parallels the classical Hamiltonian of charged particle dynamics.
Introducing kinetic momenta through a minimal substitution, we then derive a Poisson formulation for the parameter evolution. This reformulation enables the construction of a staggered-grid Boris integrator that avoids the extrapolation step required in \cite{SBHL2025}. However, the new scheme shares the same limitation, namely the lack of preservation of the symplecticity conditions for the wave packet width matrices, and therefore does not guarantee square integrability of the variational approximation. We therefore move beyond Boris-type schemes and design explicit high-order symplectic integrators based on splitting techniques and partitioned Runge--Kutta methods. These schemes conserve the quadratic invariants underlying the Hagedorn parametrization, in particular square integrability of the wave packet, preserve linear and angular momentum in the presence of symmetries, and yield near-conservation of the averaged Hamiltonian over long time intervals. Importantly, all integrators can be implemented explicitly despite the non-separable structure of the magnetic Hamiltonian.

Rigorous accuracy results are established for both the wave packet parameters and observable quantities, with error bounds that remain uniform in the semiclassical parameter. Numerical experiments illustrate the preservation of geometric structure and the improved long-time behavior compared to non-structure-preserving methods.

\subsection*{Structure of the paper}
The paper is organized as follows. \Cref{sec:mvg} introduces the magnetic Schr\"odinger equation, the Gaussian wave packet parametrization, and the continuous-time variational approximation. \Cref{sec:results} derives the Hamiltonian and Poisson formulations of the variational dynamics and establishes conservation laws. 
\Cref{sec:int} presents Boris-type and high-order symplectic integrators together with their structural properties. \Cref{sec:proofs} provides the proofs for the analysis of accuracy and energy behavior. Numerical experiments are reported in \cref{sec:num}.

%%%%%%%%%%%%%%%%%%%%%%%%%%%%%%%%%%%%%%%%%%%%%%%%%%%%%%%%%%%%
\section{Magnetic variational Gaussians}\label{sec:mvg}
We work in the Hilbert space of square integrable functions $\mathfrak H = L^2(\mathbb R^d,\mathbb C)$ and denote the inner product and the induced norm by $\langle\cdot,\cdot\rangle$ and $\Vert\cdot\Vert$, respectively. We study magnetic Schr\"odinger Hamiltonians of the form \eqref{eq:mag_ham} under the following growth and regularity assumptions. 

\begin{assumption}\label{assumption: potentials}
We suppose that the scalar potential $V : \mathbb{R} \times \mathbb{R}^d \to \mathbb{R}$ and the vector valued potential $A = (A_j)_{j=1,\ldots,d} : \mathbb{R} \times \mathbb{R}^d \to \mathbb{R}^d$ are infinitely often differentiable and satisfy the following conditions: 
        \begin{enumerate}
            \item[(a)] $V(t,\cdot)+\frac{1}{2}|A(t,\cdot)|^2$ is subquadratic, i.e., $\nabla^{\mathbf k} \left(V(t,\cdot)+\frac{1}{2}|A(t,\cdot)|^2\right)$ is bounded for all $\mathbf{k}\in\mathbb N^d$ with $\abs{\mathbf{k}} \ge 2$.
            \item[(b)] $A(t,\cdot)$ and $\partial_t A(t,\cdot)$ are sublinear, i.e., $\nabla^{\mathbf k} A(t,\cdot), 
            \nabla^{\mathbf k} \partial_t A(t,\cdot)$ are bounded for all $\mathbf{k}\in\mathbb N^d$ with $\abs{\mathbf{k}} \ge 1$.
        \end{enumerate}
\end{assumption}
        
With Assumption \ref{assumption: potentials}, magnetic Schr\"odinger dynamics are globally well-posed in~$\mathfrak H$, see e.g. \cite[\S5.3]{MasRob2017} or \cite[\S4]{Yajima}. In particular, the norm of the solution $\psi(t)$ of the Cauchy problem~\eqref{eq:Schro} with initial data $\psi_0\in \mathfrak H$ is a conserved quantity, and for time-independent Hamiltonians~$\mathcal H(t) = \mathcal H$ the energy is conserved as well, 
\[
\Vert \psi(t)\Vert = \Vert \psi_0\Vert\quad \text{and}\quad \langle \psi(t),\mathcal H\psi(t)\rangle = \langle \psi_0,\mathcal H\psi_0\rangle
\]
for all $t\in\mathbb R$.
If the Hamiltonian $\mathcal H(t)$ has further symmetries, then there are more invariants, see also \cref{sec:invariants}.

\subsection{Variational Gaussian wave packets}
We approximate the quantum solution $\psi(t)\approx u(t)$ by a complex Gaussian wave packet $u(t)\in\mathcal M$ 
with
\begin{align*}
 \mathcal{M} = &\left\{ u \in\mathfrak H \left| u(x) = \exp\left(\frac{i}{\varepsilon} \left(\frac{1}{2}(x-q)^\top  C (x-q) + p^\top  (x-q) + \zeta \right)\right), \right. \right. \nonumber\\&\hspace{.2cm}
        \left. \left. q \in \mathbb{R}^d,\, p \in \mathbb{R}^d,\, C = C^\top  \in \mathbb{C}^{d \times d},\, \Im C \text{ is positive definite},\, \zeta \in \mathbb{C} \right.\bigg\} \right..
\end{align*}
For a systematic construction of the approximation, we use the Dirac--Frenkel variational principle, 
        \begin{align}\label{eq:DF}
    i \varepsilon \partial_t u(t) = \mathcal{P}_{u(t)}(\mathcal H(t) u(t)),
    \end{align}
    where $\mathcal{P}_{u} : \mathfrak H \to \mathcal{T}_{u} \mathcal{M}$ is the orthogonal projection onto the tangent space $\mathcal{T}_{u} \mathcal{M}$, which can be identified as 
    \begin{align*}
        \mathcal{T}_u \mathcal{M} = \left\{ \varphi u \mid \varphi \text{ is a complex } d\text{-variate polynomial of degree at most } 2 \right\},
    \end{align*}
    see \cite[\S3]{LL2020}. For structure-preserving integration, it is convenient to use Hagedorn's parametrization of a normalized squeezed wave packet $u\in\mathcal M$, see \cite{Ha1980} and \cite[\S4.1]{LL2020}. We set $C = PQ^{-1}$ with $Q,P\in\mathbb C^{d\times d}$ invertible such that  
\begin{equation}\label{equation: symplecticity}
Y = \begin{pmatrix} \text{Re}\, Q & \text{Im}\, Q \\ \text{Re}\, P & \text{Im}\, P \end{pmatrix}\ \text{ is symplectic},
\end{equation}
that is, 
\[
Y^\top \Omega_d Y = \Omega_d \quad \text{for}\quad \Omega_d = \begin{pmatrix}0 & \Id_d\\ -\Id_d & 0\end{pmatrix}.
\]
Using the determinant of the complex matrix $Q$ for the normalization of the wave packet, $\Vert u \Vert = 1$, we then write 
\begin{align*}
    &u(x)=
        \frac{(\varepsilon\pi)^{-\frac{d}{4}}}{\det(Q)^{\frac{1}{2}}}\exp\left(\frac{i}{\varepsilon}\left(\frac{1}{2}(x-q)^\top PQ^{-1}(x-q+p^\top (x-q) + S\right)\right)
    \end{align*}
    with a real phase factor $S\in \mathbb{R}$. The Dirac--Frenkel principle \eqref{eq:DF} generates a globally well-posed parameter trajectory 
    $t\mapsto (q(t),p(t),Q(t),P(t),S(t))$ such that the corresponding $u(t)\in\mathcal M$ is a powerful semiclassical approximation of $\psi(t)$ with rigorous error estimates, see \cite[Theorem~4.4]{Lubich2008} and \cite[Theorem~3.10]{BDHL2023}.

\subsection{Canonical parameter coordinates}
The magnetic Schr\"odinger Hamiltonian can be seen as a semiclassically scaled pseudo-differential operator in Weyl-quantization, $\mathcal H(t) = \mathrm{op}(h(t))$, with classical symbol
\[
h(t,x,\xi) = \tfrac{1}{2}\abs{\xi - A(t,x)}^2 + V(t,x),\quad (x,\xi)\in\mathbb R^{2d},
\]
see \cref{lem:symbol}.  This allows to write the energy expectation value as a phase space integral weighted with the 
state's Wigner function,
\[
\langle \psi,\mathcal H(t)\psi\rangle = \int_{\mathbb R^{2d}} h(t,\zeta) \mathcal W_\psi(\zeta) d\zeta,\quad \psi\in\mathfrak H.
\]
For a complex Gaussian wave packet $u\in\mathcal M$, the Wigner function is a positive phase space Gaussian,
\[
\mathcal W_u(\zeta) = \frac{(2\pi)^{-d}}{\det(\Sigma_\varepsilon)} \ \exp\!\left(-\frac{1}{2}(\zeta-z)^\top \Sigma_\varepsilon^{-1}(\zeta-z)\right),\quad \zeta\in\mathbb R^{2d},
\]
with center $z=(q,p)\in\mathbb R^{2d}$ and positive definite covariance matrix $\Sigma_\varepsilon = \Yeps \Yeps^\top$ defined by
    \begin{align*}
        \Yeps = \scs Y = \begin{pmatrix}
        \scs\Re Q & \scs\Im Q\\ \scs\Re P & \scs\Im P
    \end{pmatrix},
    \end{align*}
    see \cref{proposition: wigner weyl calculus} or \cite{OhsawaTronci}. 
If we vectorize these parameters we get a 
$2D$-dimensional phase space with $2D = 2(d+2d^2)$ canonical coordinates
    \begin{align}\label{eq:coord}
        &\qB := \begin{pmatrix}
            q \\ \scs\mathrm{vec}(\Re Q) \\\scs\mathrm{vec}(\Im Q)
            \end{pmatrix}, 
        & \pB := \begin{pmatrix}
            p \\ \scs\mathrm{vec}(\Re P) \\\scs\mathrm{vec}(\Im P)
            \end{pmatrix}.
    \end{align}    
More generally, if $\mathcal{A} = \mathrm{op}(a)$ is a linear operator $\mathcal A: D(\mathcal A)\to\mathfrak H$ with smooth Weyl-symbol $a:\mathbb R^{2d}\to\mathbb C$ and if there exists $\beta_\mathbf{k} > 0$ such that 
    \begin{align}\label{equation: observable assumptions}
        \abs{\partial^\mathbf{k} a(z)} \le C_\mathbf{k}\exp(\beta_\mathbf{k} \norm{z}_2),
    \end{align}
    for all $\mathbf{k}\in \mathbb{N}^{2d}$ and $z\in\mathbb R^{2d}$, then we can write a Gaussian average $ \langle a\rangle_u := \langle u, \mathcal{A}u\rangle$ as
    \begin{align}\label{equation: weyl calculus}
        \langle a\rangle_u &=  \int_{\mathbb{R}^{2d}} a(\zeta) \mathcal{W}_u(\zeta)\, d\zeta \nonumber \\ 
        &= (2\pi)^{-d} \int_{\mathbb{R}^{2d}} a(\Yeps\zeta + z) \exp\left(-\frac{1}{2}\zeta^\top \zeta\right)\, d\zeta.
    \end{align}   
    We therefore consider the average $\langle a\rangle_u$ as a nonlinear function 
    \[
    \mathbb R^{2D}\to\mathbb C,\quad \zB = (\qB, \pB) \mapsto \langle a\rangle_u
    \]
    or equivalently $(z, \Yeps) \mapsto \langle a\rangle_u$; the last equality in the definition of the average shows that their $\varepsilon$-dependence comes only through the covariance factor $\Yeps$. 

    %%%%%%%%%%%%%%%%%%%%%%%%%%%%%%%%%%%%%%%%%%%%%%%%%%%%%%%%%%%%%%%%%%
    \subsection{Linear and quadratic invariants}\label{sec:invariants}

If the potentials $A(t,\cdot)$ and $V(t,\cdot)$ have translational (or rotational) symmetries, then the classical system 
\[
\dot q = p-A,\quad \dot p = \nabla_q A^\top p-\nabla_q V
\]
conserves the total linear (or angular) momentum; this generalizes to the variational parameters as shown for the electric Schrödinger equation in \cite[\S4]{FL2006} and extended to the magnetic case in \cite[\S4]{O2015}. There, also a semiclassical angular momentum is introduced, see 
    \cite[eq.(27)]{O2015}, 
    \[
        L_\varepsilon = pq^\top - qp^\top + \frac{\varepsilon}{2}\Re(PQ^* - QP^*).
    \]
    All this can be formulated in terms of the vectorized coordinates $(\qB, \pB)$, if the potentials are invariant as follows.
      
\paragraph{Translation invariant} Then the linear momentum $\sum_{j=1}^d p_j$ is conserved, which can be written as a linear invariant in $\pB$.

\paragraph{Rotation invariant} Then the semiclassical angular momentum $L_\varepsilon$ is conserved, which can be written as a mixed quadratic invariant in $\qB$ and $\pB$. Indeed, $L_\varepsilon$ is uniquely defined by its action on skew-symmetric matrices $K\in\mathbb R^{d\times d}$, 
        \begin{align*}
            L_\varepsilon (K) &= \frac{1}{2}\tr(L_\varepsilon^\top K) = p^\top K q + \frac{\varepsilon}{2}\Re \left( \tr(P^* K Q)\right)\\& = p^\top K q + \frac{\varepsilon}{2}\Re \left(\mathrm{vec}(P)^* (\Id_d\otimes K)\mathrm{vec}(Q)\right)\\
        &= 
                \pB^\top \begin{pmatrix}
                    K & & \\
                    & \Id_d\otimes K & \\
                    & & & \Id_d\otimes K
                \end{pmatrix}\qB,
            \end{align*}
            where the last equation uses  Kronecker product calculus, see e.g. \cite[Ch. 2]{Z2013}.

\smallskip
Moreover, the symplecticity condition \eqref{equation: symplecticity} for the matrix $Y$ can also be viewed as a mixed quadratic form. For example, examining its top-left block, we write 
    $
    E_{(j,k)} = (e_je_k^\top -e_ke_j^\top )\otimes\Id_{d}
    $
    with $j,k=1,\ldots,d$ to get 
        \begin{align*}
            &\Re(Q)^\top \Re(P)-\Re(P)^\top \Re(Q))_{jk} \\
            &=\left(\hspace{-2mm}\begin{array}{ll}
               \mathrm{vec}(\Re Q)^\top   &  \mathrm{vec}(\Im Q) ^\top 
            \end{array}\hspace{-2mm}\right)
            \left(\hspace{-2mm}\begin{array}{cc}
                E_{(j,k)} &  \\
                 & -E_{(j,k)}
            \end{array}\hspace{-2mm}\right)
            \left(\hspace{-2mm}
            \begin{array}{cc}
               \mathrm{vec}(\Re P)  \\  \mathrm{vec}(\Im P) 
            \end{array}\hspace{-2mm}\right).
        \end{align*}
    Rewriting the other blocks similarly, we can define a family of matrices $I_k\in\mathbb R^{D\times D}$ such that 
       \eqref{equation: symplecticity} is equivalent to
       \begin{equation}\label{equation: quadratic}
       \qB^\top I_k\pB = \mathrm{vec}(\Omega)_k,\quad k=1,\ldots,4d^2.
       \end{equation}

%%%%%%%%%%%%%%%%%%%%%%%%%%%%%%%%%%%%%%%%%%%%%%%%%%%%%%%%%%%%%%%%%%%%%%%%
\section{Continuous-time results}\label{sec:results}
We start by presenting our key observation on derivatives of averages with respect to the vectorized parameters $z$ and $\Yeps$, that govern the phase space center and the complex width of a Gaussian $u\in\mathcal M$. 
We have the following two striking relations.
    \begin{theorem}[Derivative formulas]\label{proposition: derivative average}
        Let $\mathcal{A} = \mathrm{op}(a)$ be an observable with smooth symbol $a$ that satisfies the growth assumption \eqref{equation: observable assumptions}, then
        \begin{align*}
            \nabla_z \langle a \rangle_{u}= \langle\nabla_\zeta  a\rangle_{u}\ \text{ and } \ \nabla_{\Yeps} \langle a \rangle_{u}= \langle\nabla_\zeta \nabla_\zeta ^\top  a\rangle_{u}\Yeps.
        \end{align*}
    \end{theorem}

\begin{proof} 
The proof can be found in \cref{proof:av}.
\end{proof}
This is the primary tool for deriving the Hamiltonian structure, and it is also practically helpful, as we used it with automatic differentiation routines to provide a fully flexible wave packet integrator for the numerical experiments in \cref{sec:num}.

\subsection{General Hamiltonian dynamics}
The combination of the above two derivative identities with the canonical parametrization $\zB = (\qB,\pB)$ introduced in \eqref{eq:coord} reveals several structural properties of the variational approximation that are not immediately apparent. 
For their development, we adopt a general Hamiltonian perspective and view the magnetic Schr\"odinger operator \eqref{eq:mag_ham} as a special case of a  linear operator satisfying the following set of assumptions. 

\begin{assumption}\label{assumption: assumption symbol}
        We suppose that the Weyl-symbol $h(t,\cdot)$ of the Hamiltonian operator $\mathcal{H}(t)=\mathrm{op}(h(t,\cdot))$ is smooth, satisfies the exponential growth assumption \eqref{equation: observable assumptions}, and subquadratic growth, i.e., there exist $C_\mathbf{k}(t)\in L_\mathrm{loc}^1$ such that $\abs{\partial_z^\mathbf{k} h(t, z)}\le C_\mathbf{k}(t)$ for $z\in \mathbb{R}^{2d}$ and $\abs{\mathbf{k}} \ge 2$.
    \end{assumption}

Under these assumptions, the Cauchy problem \eqref{eq:Schro} is well-posed in $\mathfrak H$, see e.g. \cite[\S5.6]{MasRob2017}. For such general Hamiltonian operators, \cite[\S3]{BDHL2023} established well-posedness of the variational approximation $u(t)\in\mathcal M$ and derived the equations of motion for the parameters $(q,p,Q,P,S)$; using the canonical parameter coordinates $(\qB,\pB)$ introduced in \eqref{eq:coord}, these results can be recovered in a concise and transparent manner. These well-posedness results can also be established for Hamiltonians with coercive symbols, due to the existence of the classical flow.
    
    \begin{corollary}[Equations of motion]\label{theorem: symplectic evolution}
        Under Assumption \ref{assumption: assumption symbol} consider the variational wave packet approximation $u(t)\in\mathcal M$. Then, 
        the equations of motion for the canonical parameters $\zB(t)$ are given by Hamilton's equations 
        \begin{align}\label{equation: equations of motion vectorized}
            \dot\zB(t) = \Omega_D \nabla_{\zB}\hB(t,\zB(t)).
            %\frac{d}{dt}\begin{pmatrix}
            %    \qB(t)\\\pB(t)
            %\end{pmatrix} = \begin{pmatrix}
            %    0 & \Id_D\\ -\Id_{D} & 0
            %\end{pmatrix} \begin{pmatrix}
            %    \nabla_{\qB} \hB(t, \qB(t), \pB(t))\\ \nabla_{\pB} \hB(t, \qB(t), \pB(t))
            %\end{pmatrix}.
        \end{align} 
        for the averaged Hamiltonian $\hB(t, \zB(t)) := \langle h(t,\cdot)\rangle_{u(t)}$.
    \end{corollary}

    \begin{proof}
        The proof can be found in \cref{sec:proofs_gen}.
    \end{proof}
    
    \begin{remark}
        The phase $S(t)$ of the wave packet $u(t)$ evolves similarly to the classical action of the system: Using the equation derived in \cite[eq. 3.4d]{BDHL2023} and the phase change of the determinant, we get 
        \begin{align}\label{equation: phase de}
            \frac{d}{dt}\left(S(t) + \frac{\varepsilon}{4}\Re \tr(P(t)Q(t)^*)\right) = \pB(t)^\top \dot{\qB}(t) - \hB(t).
        \end{align}
    \end{remark}

    Global well-posedness of the variational dynamics is straightforward due to the following Corollary.

\begin{corollary}[Well-posedness]\label{proposition: global well-posedness}
Under Assumption \ref{assumption: assumption symbol} consider the variational wave packet approximation $u(t)\in\mathcal M$. Then, the variational Hamiltonian 
\[
\hB(t,\cdot):\mathbb R^{2D}\to\mathbb R, \quad \hB(t,\cdot) = \langle h(t,\cdot)\rangle_{u(t)}
\]
inherits the subquadratic growth or the coercivity of the classical Hamiltonian symbol $h(t,\cdot):\mathbb R^{2d}\to\mathbb R$. In particular, the variational equations of motion \eqref{equation: equations of motion vectorized} are globally well-posed.
\end{corollary}

    \begin{proof}
        The proof can be found in \cref{sec:proofs_gen}.
    \end{proof}

    %%%%%%%%%%%%%%%%%%%%%%%%%%%%%%%%%%%
    \subsection{Magnetic dynamics} \label{sec:mag_dyn}
    For the magnetic Hamiltonian \eqref{eq:mag_ham}, the variational Hamiltonian of \cref{theorem: symplectic evolution} takes the form 
        \begin{align}\label{equation: effective hamiltonian}
        \hB (t,\qB ,\pB )=\frac{1}{2}\pB ^\top \pB -\pB ^\top \AB (t,\qB )+\VB (t,\qB ),
    \end{align}
    with averaged potentials 
    \begin{align}\label{eq:pot_mag}
            &\AB(t,\qB) := \begin{pmatrix}
                \langle A(t,\cdot)\rangle_u \\ \mathrm{vec}\left(\langle J_A(t,\cdot)\rangle_u\scs\Re Q\right) \\\mathrm{vec}\left(\langle J_A(t,\cdot)\rangle_u\scs\Im Q\right)
            \end{pmatrix},\\*[1ex]\nonumber
            &\VB(t,\qB) := \tfrac12 \langle|A(t,\cdot)|^2\rangle_u + \langle V(t,\cdot)\rangle_u,
    \end{align}
    where $J_A(t,x)$ denotes the Jacobi matrix of $A(t,x)$, see \cref{lem:symbol}.    The variational equations of motion are thus given by 
    \begin{align}\label{equation: equations of motion magnetic}
        \dot{\qB}(t) = \pB(t) - \AB(t,\qB(t)), \quad \dot{\pB}(t) = J_{\AB}(t,\qB(t))^\top\pB(t) - \nabla_{\qB}\VB(\qB(t)),
    \end{align}
    with Jacobian $J_{\AB}(t,\qB) = (\nabla_{\qB} \AB(t,\qB)^\top)^\top$, see also \cite[\S3]{BDHL2023}. In particular, if the vector potential $A$ is linear in $x$, $A(t,x) = M_A(t)x$ for some $M_A(t)\in\mathbb R^{d\times d}$, then
    \begin{align*}
    &\AB(t,\qB) = \begin{pmatrix}
        M_A(t)&&\\&\Id_d\otimes M_A(t)&\\&&\Id_d\otimes M_A(t)
    \end{pmatrix}\qB,\\
    %\end{align*} 
    %\begin{align*}
    %        &\AB(t,\qB) = \begin{pmatrix}
    %            M_A(t)\qB \\ \mathrm{vec}\left(\langle M_A(t)\scs\Re Q\right) \\\mathrm{vec}\left(M_A(t)\scs\Im Q\right)
    %        \end{pmatrix},\\*[1ex] 
            &\VB(t,\qB) = \frac{1}{2}\AB(t,\qB)^\top \AB(t,\qB) + \langle V(t,\cdot)\rangle_u.
    \end{align*}
    The variational system retains even more properties of classical charged particle dynamics; we can perform a minimal substitution from canonical to kinetic momenta and obtain analogous equations of motion.
    
    \begin{corollary}[\cite{SBHL2025}, Theorem~4.1]\label{corollary: equations of motion charged particle}
    Under Assumption \ref{assumption: potentials} consider the variational wave packet approximation $u(t)\in\mathcal M$. 
    Introduce the kinetic momentum 
    \begin{align}\label{equation: kinetic momenta}
            \vB (t)=\pB (t)-\AB (t,\qB (t)).
        \end{align}
    Then, the averaged Hamiltonian, written as a function of $(t,\qB,\vB)$, becomes
    \begin{align*}
            \hB(t,\qB ,\vB ) = \frac{1}{2} \vB ^\top \vB  + \bar{\VB }(t,\qB )
    \end{align*}
    with $\bar{\VB }(t,\qB ) = \VB(t,\qB)-\frac{1}{2}\AB(t,\qB)^\top \AB(t,\qB)$. 
    The equations of motion for the parameters $(\qB(t), \vB(t))$ are given by 
        \begin{align}\label{equation: equations of motion charged particle}
            \dot{\qB }(t) = \vB (t),\quad
            \dot{\vB }(t) = -\BB (t,\qB (t)) \vB (t)+\EB (t,\qB (t)),
        \end{align}
        where we write
        \begin{align*}
            \BB (t,\qB )&=J_{\AB}(t,\qB) - J_{\AB}(t,\qB)^\top,\\
            \EB (t,\qB )&=-\nabla_{\qB}\bar{\VB }(t,\qB )-\partial_t\AB (t,\qB ).
        \end{align*}
        If $A(t,x) = M_A(t)x$ for some $M_A(t)\in\mathbb R^{d\times d}$, then 
        $\BB(t,\qB) = \Id_{1+2d} \otimes B(t)$ with $B(t) = M_A(t)-M_A(t)^\top$ and $\bar{\VB}(t,\qB) = \langle V(t,\cdot)\rangle_{u}$.
    \end{corollary} 

    \begin{proof}
        The proof can be found in \cref{sec:proofs_gen}.
    \end{proof}

The vectorized kinetic momentum $\vB$ introduced in \eqref{equation: kinetic momenta} contains the shifted matrix $P-\langle J_A\rangle_u$. For preservation of the symplecticity condition \eqref{equation: symplecticity}, which guarantees the square integrability of the wave packet, we will use the following characterization with $M = \langle J_A\rangle_u$ later on.

\begin{lemma}[Invariants for kinetic momenta]\label{proposition: magnetic Invariants}
Let $M\in\mathbb R^{d\times d}$. Two matrices $Q,P\in\mathbb C^{d\times d}$ define a matrix $Y\in\mathbb R^{2d\times 2d}$ satisfying the symplecticity condition \eqref{equation: symplecticity} if and only if the matrix 
\begin{align*}
            Y_M = \begin{pmatrix}
                \Re Q & \Im Q\\ \Re P-M \Re Q & \Im P-M \Im Q
            \end{pmatrix} 
            %\left(\Id + \Omega \begin{pmatrix}
            %    \langle J_A\rangle_u & 0\\ 0 & 0
            %\end{pmatrix}\right)Y
        \end{align*}
        satisfies 
        \begin{align}\label{equation: magnetic symplecticity}
            Y_M^\top
            %\underbrace{
            \begin{pmatrix}
                M-M^\top &\Id_d\\-\Id_d & 0
            \end{pmatrix}
            %}_{=:\Omega_B}
            Y_M = \Omega_d.
        \end{align}
    \end{lemma}

    \begin{proof}
        This can be proven by direct computation.      
    \end{proof}

%%%%%%%%%%%%%%%%%%%%%%%%%%%%%%%%%%%%%%%%%%%%%%
\section{Time-integrators}\label{sec:int}
We now construct a Boris integrator and higher-order splitting methods, aiming at the preservation of the geometric structure of the 
variational dynamics. For notational simplicity, we consider time-independent fields. For the more general time-dependent case, 
the fields shall be evaluated at time $t$ when updating from $t$ to $t+\stepsize$ with step size $\stepsize>0$, see \cite{HL2018}.

\subsection{Boris-type time integration}
The Boris algorithm for classical charged particle dynamics $\dot q = v$, $\dot v = -B(q) v + E(q)$ works with 
$q_n\approx q(t_n)$ and $v_n\approx v(t_n-\frac\stepsize2)$ on a staggered time grid, setting
\[
\frac{q_{n+1}-q_n}{\stepsize} = v_{n+1},\quad
\frac{v_{n+1}-v_n}{\stepsize} = - B(q_n) \frac{v_{n+1}+v_n}{2} + E(q_n),\quad n\ge 0. 
\]
Turning this into an explicit time integrator, the equation for the kinetic momenta may be rewritten as 
\[
(\Id_3 + \hat\Omega_n) v_{n+1} = (\Id_3-\hat\Omega_n) v_n + \stepsize E(q_n)\quad \text{with}\quad 
\hat\Omega_n = \frac\stepsize2 B(q_n),
\]
see \cite[eq.(9)]{QZXLST2013}. We apply this formulation to the variational equations of motion of 
\cref{corollary: equations of motion charged particle} 
and define the one-step map 
        \begin{align}\label{equation: Boris one-step map}
            &\Psi_{\stepsize}^B:\mathbb{R}^{2D} \to \mathbb{R}^{2D},\nonumber \\&\quad \begin{pmatrix}
                \qB \\\vB ^s
            \end{pmatrix}  \mapsto \begin{pmatrix}
                 \qB  + \stepsize \left(R\left(\frac{\stepsize }{2}\BB (\qB )\right)\vB ^s + \stepsize \left(\Id+\frac{\stepsize }{2}\BB (\qB )\right)^{-1} \EB (\qB )\right)\\ R\left(\frac{\stepsize }{2}\BB (\qB )\right)\vB ^s + \stepsize \left(\Id+\frac{\stepsize }{2}\BB (\qB )\right)^{-1} \EB (\qB )
            \end{pmatrix},
        \end{align}
        where 
        \[
        R(\tfrac\stepsize2 \BB(\qB)) = (\Id_D + \tfrac\stepsize2\BB(\qB))^{-1}(\Id_D - \tfrac\stepsize2\BB(\qB))
        \] 
        is the Cayley transform of the skew-symmetric matrix $\frac\stepsize2\BB(\qB)$. 
        The method is now applied with initial conditions $\qB _0,\vB _0$ and a first step $\vB _1^s = \vB _0 -\frac{\stepsize }{2}\BB (\qB _0)\vB_0 + \frac{\stepsize }{2}\EB (\qB _0)$.
        The numerical trajectory $(\qB _n,\vB _n^s)_{n\le n^*}$ is given by $(\qB _{k+1},\vB _{k+1}^s) = \Psi_{\stepsize}^B(\qB _k,\vB _k^s)$ and an approximation for $\vB _k$ is given by $\frac{\vB _{k+1}^s+\vB _k^s}{2}$.
%    \end{definition}
The structural advantages of the Boris algorithm, as described in \cite{QZXLST2013}, rely on the Cayley transform and hold for this integrator as well. 
    
    \begin{remark}
        The one-step map \eqref{equation: Boris one-step map} defines the same Boris-type method as \cite{SBHL2025} only for the case of a linear vector potential $A$. In general, the two methods differ, since \cite{SBHL2025} applies the Boris approach to the first-order $A$ derivatives in $\BB$ and treats higher-order derivatives by extrapolation.
    \end{remark}

Rewriting the momentum update of \eqref{equation: Boris one-step map} as
\[
R\left(\frac{\stepsize }{2}\BB (\qB )\right)\vB ^s + \stepsize \left(\Id+\frac{\stepsize }{2}\BB (\qB )\right)^{-1} \EB (\qB )
 = R\left(\frac{\stepsize }{2}\BB(\qB)\right)\left(\vB^s + \frac{\stepsize }{2}\EB(\qB)\right) + \frac{\stepsize }{2}\EB(\qB)
\]
we obtain $\Psi_{\stepsize} ^B=\Psi_{\stepsize}^\mathrm{kin}\circ \Psi_{\frac{\stepsize }{2}}^\mathrm{pot} \circ \Psi_{\stepsize}^\mathrm{mag} \circ \Psi_{\frac{\stepsize }{2}}^\mathrm{pot}$ with sub-steps 
    \begin{align*}
        &\Psi_{\stepsize}^\mathrm{kin}\begin{pmatrix}
            \qB\\\vB^s
        \end{pmatrix} = \begin{pmatrix}
            \qB + \stepsize  \vB^s\\\vB^s
        \end{pmatrix}, \quad 
        \Psi_{\stepsize}^\mathrm{pot}\begin{pmatrix}
            \qB\\\vB^s
        \end{pmatrix} = \begin{pmatrix}
            \qB\\\vB^s + \stepsize  \EB(\qB)
        \end{pmatrix},\\
        &\Psi_{\stepsize}^\mathrm{mag}\begin{pmatrix}
            \qB\\\vB^s
        \end{pmatrix} = \begin{pmatrix}
            \qB \\R\left(\frac{\stepsize }{2}\BB(\qB)\right)\vB^s
        \end{pmatrix},
    \end{align*}
    where $\Psi_{\stepsize}^\mathrm{mag}$ approximates the magnetic evolution by the implicit midpoint rule. 
    This motivates the definition of a symmetric splitting integrator,
    \begin{align}\label{equation:  splitting Boris integrator}
            &\Psi_{\stepsize}^S = \Psi_{\frac{\stepsize }{2}}^\mathrm{kin}\circ \Psi_{\frac{\stepsize }{2}}^\mathrm{pot} \circ \Psi_{\stepsize}^\mathrm{mag} \circ \Psi_{\frac{\stepsize }{2}}^\mathrm{pot} \circ \Psi_{\frac{\stepsize }{2}}^\mathrm{kin},
        \end{align}
    that agrees with the Boris formulation up to reordering of the individual updates and the initial step. 
    For this Boris integrator, we have a near-conservation result for the case of magnetic dynamics with linear vector potential. 
    The crucial symplecticity condition \eqref{equation: symplecticity}, that guarantees the square integrability of the wave packet, is preserved up to second order with an explicit remainder term.

    \begin{proposition} [Near-conservation]\label{proposition: modified magnetic invariant}
        Suppose that the potentials $A$ and $V$ do not depend on time and that $A(x) = Mx$ for some $M\in\mathbb R^{d\times d}$. 
        Consider $\zB_0\in\mathbb R^{2D}$ and $\zB_1 = \Psi_{\stepsize}^S(\zB_0)$ for $\stepsize>0$, and 
        define associated matrices $Y_0$ and $Y_1$ in $\mathbb R^{2d\times 2d}$ from reshaping the corresponding 
        entries into matrix form. Then, 
        \[
        Y_1^\top \Omega_B(\tau) Y_1 = Y_0^\top \Omega_B(\tau) Y_0,
        \]
        where 
        \[
        \Omega_B(\tau) = \begin{pmatrix}
                B &\Id_d\\-\Id_d & -\frac{\stepsize ^2}{4}B
            \end{pmatrix}, \quad B = M-M^\top,
        \]
        differs from the structure matrix of \cref{proposition: magnetic Invariants} in the lower right block by $-\frac{\stepsize ^2}{4}B$. 
    \end{proposition}

    \begin{proof}
        The proof can be found in \cref{sec:proof_Boris}.
    \end{proof}

    %%%%%%%%%%%%%%%%%%%%%%%%%%%%%%
    \subsection{Symplectic time integration}\label{sec:int_symp}
    To ensure exact conservation of the matrix invariants,
    we want to design explicit, symplectic integration schemes of arbitrary high order. We work with the canonical parameter coordinates $(\qB,\pB)$ and the variational Hamiltonian $\hB(t,\qB,\pB)$ defined in \eqref{equation: effective hamiltonian}. 
    If the potentials $A$ and $V$ are explicitly time-dependent, then we first apply a commutator-free Magnus integrator of the desired order.
    The resulting flow is a concatenation of the flows of weighted, time-averages of $\hB(t,\cdot)$, e.g. mid-point approximation for order two or \cite[Thm. 3.1]{BM2001} for order four. Thus, we can reduce our investigation here to the case of time-independent potentials. We write the variational Hamiltonian \eqref{equation: effective hamiltonian} as the sum 
    \[
    \hB  = \hB _\mathrm{kin} +  \hB _\mathrm{pot} +  \hB _\mathrm{mag}
    \]
    with $\hB _\mathrm{kin} = \frac{1}{2} \pB ^\top \pB$, $\hB _\mathrm{pot} = \VB(\qB)$, and 
        $\hB _\mathrm{mag} = -\AB(\qB) ^\top \pB$. Now, the equations for $\hB_\mathrm{kin}$ and $\hB_\mathrm{pot}$ can be integrated exactly, $\dot{\qB } = \pB,\, \dot{\pB } = 0 \ \text{ and }\  \dot{\qB } = 0,\, \dot{\pB } = -\nabla_{\qB} \VB(\qB),$ and the difficulty only lies in the integration of the equations for the non-separable Hamiltonian $\hB _\mathrm{mag}$,
    \begin{align}\label{equation: magnetic hamiltonian evolution}
        \dot{\qB} &= -\AB(\qB) , \quad
        \dot{\pB} = J_{\AB}^\top (\qB)  \pB.
    \end{align}
    For this task, we propose a partitioned Runge--Kutta method of the following form. We consider an explicit $s$-stage Runge--Kutta method $(L,b)$ of order $\nu$ with weights $b_i\neq 0$ and define a second Butcher tableau $(\hat L,b)$ with
        \begin{align}\label{equation: symplectic partition}
            \hat{L} := \mathbf{1}_s b^\top-\mathrm{diag}(b)^{-1}L^\top\mathrm{diag}(b),
        \end{align} 
    where we write $\mathbf{1}_s$ for the $s$-dimensional vector of ones and $\mathrm{diag}(b)$ for the diagonal matrix with diagonal $b$. Then, 
    \begin{equation}\label{eq:PRK}
    \qB_{n+1} = \qB + \stepsize \sum_{i=1}^s b_i \qB_{n+1}^{(i)},\quad \pB_{n+1} = \pB + \stepsize \sum_{i=1}^s b_i \pB_{n+1}^{(i)}
    \end{equation}
    with stages    
    \begin{align}\label{equation: PRK stages}
        \qB_{n+1}^{(i)} = - \AB \hspace{-1.4mm}\left(\qB_{n}  + \stepsize \sum_{j=1}^{i-1} l_{ij}\qB_{n}^{(j)}\right),\quad
        \pB_{n+1}^{(i)} = M_n^{(i)} \left(\pB_{n}  + \stepsize\sum_{j=1}^s \hat{l}_{ij}\pB_{n}^{(j)}\right),
    \end{align}
    where 
    \begin{align}\label{equation: RK rhs}
        M_n^{(i)} := J_{\AB}^\top \hspace{-1.4mm}\left(\qB_{n}  + \stepsize \sum_{j=1}^{i-1} l_{ij}\qB_{n}^{(j)}\right)\in\mathbb R^{D\times D}.
    \end{align}

\begin{remark}\label{rem:quad}
A partitioned Runge--Kutta method $(L,b)$, $(\hat L,\hat b)$ with $b=\hat b$ conserves quadratic invariants if and only if 
$b_i \hat L_{ij}  + \hat b_j L_{ji} = b_i\hat b_j$ for all $i,j$, see for example \cite[Thm. IV.2.4]{HLW2006}. 
This condition has motivated our choice of $\hat L$ in \eqref{equation: symplectic partition}.
\end{remark}
    
   We denote by $M_n\in\mathbb R^{sD\times sD}$ the block-diagonal matrix with diagonal blocks $M_n^{(i)}$ and write the $s$ momentum stages $\pB_{n}^{(i)}$ vectorized as $k_n = M_n(\mathbf{1}_s \otimes \Id_{D})\pB_n + \stepsize M_n (\hat{L} \otimes \Id_{D}) k_n$. Then, $\pB_{n+1} = \hat{R}(\stepsize M_n)\pB_n$ with
        \begin{align}\label{eq:stability}
            \hat{R}(\stepsize M_n) = 
            \Id_{D} + \stepsize (b^\top\otimes \Id_{D}) W_n^{-1}M_n(\mathbf{1}_s \otimes \Id_{D}),
        \end{align}
        if $W_n = \Id_{sD}-\stepsize M_n (\hat{L} \otimes \Id_D)$ is invertible. Next, we reformulate this momentum one-step map to avoid the inverse of the $sD\times sD$ matrix $W_n$, thereby making both the analysis and the explicit implementation of the integrator easier and cheaper.
     
    \begin{proposition}[Momentum one-step map]\label{prop:stability_function}
        Let $L\in\mathbb R^{s\times s}$ be a strictly lower triangular matrix and $b\in\mathbb R^s$ such that $b_i\neq0$ for all $i$. Then, 
        the generalized stability function 
        \[
        \hat{R}(\stepsize M_n) = 
            \Id_{D} + (b^\top\otimes \Id_{D}) W_n^{-1}\stepsize M_n(\mathbf{1}_s \otimes \Id_{D})
            \] 
        defined in \eqref{eq:stability} 
        can be re-expressed as
        \begin{align*}
            \hat{R}(\stepsize M_n)&= \left(\Id_D + \rho\!\left(\stepsize M_n^{(1)},\ldots,\stepsize M_n^{(s)}\right)\right)^{-1},
        \end{align*}
        if the inverse exists. Here, $\rho$ is a multi-variate polynomial such that for all $x\in\mathbb R^s$
        \[
        \rho\left(\stepsize x_1,\ldots,\stepsize x_s\right) = \sum_{k=1}^s (-\stepsize)^k \left(\sum_{1\le j_1<\ldots<j_k\le s} x_{j_1} \cdots x_{j_k} b_{j_k}\prod_{\ell=1}^{k-1}L_{j_{\ell+1} j_\ell}\right).
        \]
    \end{proposition}

    \begin{proof}
    The proof can be found in \cref{section: explicit implementation}.
    \end{proof}

Owing to the above reformulation of the stability function, we can now directly establish the well-definedness and the order of the partitioned Runge--Kutta  
method.
   
    \begin{theorem}[Magnetic discretization]\label{theo:mag_disc}
        We consider an $s$-stage explicit Runge--Kutta method $(L,b)$ of order $\nu$ with $b_i\neq 0$ for all $i=1,\ldots,s$,  
        and apply the partitioned method with Butcher tableau $(L,b)$, $(\hat L,b)$ to the evolution of $\hB_\mathrm{mag}$. 
        If the numerical trajectory stays in a compact set $K$, then there is a constant $\stepsize_0 > 0$ that depends only on $L, b, K$ and the magnetic potential $A$ such that for step sizes $\stepsize \le \stepsize_0$, the integrator is well-defined, explicit, symplectic, and of order $\nu$.
    \end{theorem}
    
    \begin{proof}
    The proof can be found in \cref{section: proofs magnetic integrator}.
    \end{proof}

    The following corollary follows immediately, since a splitting method inherits well-posedness, symplecticity, and order from its sub-steps.
    
    \begin{corollary}[Full magnetic discretization]\label{theorem: symplectic integrator}
        Suppose we are given a splitting scheme of order $\nu$, an explicit $s$-stage Runge--Kutta method $(L,b)$ of order $\nu$ with weights $b_i\neq 0$ for all $i=1,\ldots,s$. We apply
        the partitioned Runge--Kutta method with Butcher tableau $(L,b)$, $(\hat L,b)$ to the evolution of $\hB _\mathrm{mag}$ and consider the resulting splitting integrator for $\hB = \hB _\mathrm{kin} +  \hB _\mathrm{pot} +  \hB _\mathrm{mag}$.
        If the numerical trajectory stays in a compact set $K$, then there is a constant $\stepsize_0 > 0$ that depends only on $L, b, K$ and $A, V$ such that for step sizes $\stepsize \le \stepsize_0$, the integrator is well-defined, explicit, symplectic and of order $\nu$. 
    \end{corollary}

    \begin{corollary}[Conservation properties] \label{cor:conserve} The method of \cref{theorem: symplectic integrator}
        conserves the  symplecticity condition \eqref{equation: symplecticity} for the wave packet's width matrix; if the potentials $A$ and $V$ are symmetric with respect to translation or rotation, then it also conserves the linear or semiclassical angular momentum.         
    \end{corollary}

    \begin{proof}
        The proof can be found in \cref{section: proofs magnetic integrator}.
    \end{proof}

    %%%%%%%%%%%%%%%%%%%%%%%%%%%%%%%%%%%%%%%%%%%%%%%%%%%%%%%%%%%%%%%%%%%%%%%%%%%%%%%%%%%%%%%%%%%%%%%%%%%%
    \subsection{Accuracy analysis and near-conservation of the energy}\label{section: error analysis}
    
    We consider a parameter trajectory $(\zB_n)_{n\le n^*} = (\qB_n, \pB_n)_{n\le n^*}$ obtained from a symplectic, order $\nu$ integrator that preserves the symplecticity condition \eqref{equation: symplecticity}, possibly the integrator of \cref{theorem: symplectic integrator}. Furthermore, we assume that the corresponding phases $(S_n)_{n\le n^*}$ are obtained with order $\nu$ by integrating equation \eqref{equation: phase de}. Then, the corresponding wave packet satisfies at times $t_n = n\stepsize$
    \begin{align*}
            \norm{u(\zB_n, S_n) - u(\zB(t_n), S(t_n))}_2 \le C\,\frac{\stepsize^\nu}{\varepsilon} \quad \text{ for }n\le n^*.
    \end{align*}
    The constant $C>0$ may depend exponentially on $t_{n^*}$, but is independent of $n, \stepsize$ and the semiclassical parameter $\varepsilon$, 
    see \cite[Thm. 7.7]{LL2020} and apply literally the same proof. As for observable accuracy we slightly extend the result of \cite[Thm. 7.7]{LL2020}.

    \begin{theorem}[Observable accuracy]\label{theorem: accuracy analysis}
        If $\mathcal A = \mathrm{op}(a)$ is an observable with smooth symbol $a$ that is globally Lipschitz with constant $L_a>0$, then the error of the approximate average is bounded by 
        \begin{align*}
            \abs{\langle a\rangle_{u}(\zB_n) - \langle a\rangle_{u}(\zB(t_n))} \le L_a \widetilde C \stepsize^\nu \quad \text{ for }n\le n^*.
        \end{align*}
        The constant $\widetilde C>0$ may depend exponentially on $t_{n^*}$, but is otherwise independent of $n, \stepsize$ and $\varepsilon$.
    \end{theorem}

    \begin{proof}
        We have $|a(\zeta)|\le |a(0)| + L_a|\zeta|$ for all $\zeta\in\mathbb R^{2d}$, so that the phase space integrals are convergent. Moreover, 
        \begin{align*}
            &\abs{\langle a\rangle_{u}(\zB_n) - \langle a\rangle_{u}(\zB(t_n))} \\& \hspace{1cm}\le (2\pi)^{-d}\int_{\mathbb{R}^{2d}}\abs{a(\Yeps_n\zeta + z_n) - a(\Yeps(t_n)\zeta + z(t_n))}\exp(-\frac{1}{2}\zeta^\top\zeta)\,d\zeta \\& \hspace{1cm}\le (2\pi)^{-d}\int_{\mathbb{R}^{2d}}L_a\left(\norm{(\Yeps_n-\Yeps(t_n))\zeta}_2 + \norm{z_n-z(t_n)}_2\right)\exp(-\frac{1}{2}\zeta^\top\zeta)\,d\zeta \\& \hspace{1cm}\le L_a \left(C_d \norm{\Yeps_n-\Yeps(t_n)}_2 + \norm{z_n-z(t_n)}_2\right) \le 
            L_a c \stepsize^\nu(C_d + 1),
        \end{align*}
        where $C_d>0$ bounds the first absolute moment of a Gaussian and depends only on the dimension.
    \end{proof}

    For a proof of near-conservation of the energy, we assume that the symbol extends to an entire function $h:\mathbb{C}^{2d} \to \mathbb{C}$ with at most exponential growth at infinity. Under these assumptions, the averaged Hamiltonian $\hB(\zB) = \langle h\rangle_{u}(\zB)$ allows the direct application of standard results for long-time energy conservation.

    \begin{theorem}\label{theorem: near-cons energy}
    Consider a time-independent Hamiltonian $H=\mathrm{op}(h)$ with entire symbol $h$, that satisfies Assumption \ref{assumption: assumption symbol}.
    Suppose that the numerical trajectory $(\zB_n)_{n\le n^*}$ given by an order $\nu$ integrator from \cref{theorem: symplectic integrator} with step size $\stepsize$ stays in a compact set $K$. Then, there exists $\stepsize_0>0$ such that 
    \begin{align}
        \abs{\langle h\rangle_u(\zB_n) - \langle h\rangle_u(\zB_0)} = \abs{\hB(\zB_n) - \hB(\zB_0)} \le C_h\stepsize^\nu
    \end{align}
    for $n\stepsize \le \exp\left(\frac{\stepsize_0}{2\stepsize}\right)$ and $n\le n^*$. The error constant $C_h > 0$, as well as $\stepsize_0$, depend only on $K$ and the growth of $h$, but are independent of $\stepsize, n^*$.
    In particular, the error does not depend on $\varepsilon$ if the numerical solution can be bounded independently of it.
    \begin{proof}
        We establish boundedness and analyticity of the average Hamiltonian $\hB(\zB)$ in \cref{sec:prooferr} and then directly apply 
        \cite[Thm. 8.1]{HLW2006}.
    \end{proof}
    \end{theorem}

%%%%%%%%%%%%%%%%%%%%%%%%%%%%%%%%%%%%%%
\section{Proofs}\label{sec:proofs}
In this section, we present the postponed proofs of our main results.

%%%%%%%%%%%%%%%%%%%%%%%%%%%%%%%%%%%%%%%%%%%%%%%%%%%%
\subsection{Derivatives of averages}\label{proof:av}
        \begin{proof}[Proof of \cref{proposition: derivative average}]
        We crucially use \eqref{equation: weyl calculus} and the growth assumptions on $a$. We perform integration by parts, 
            \begin{align*}
                \nabla_z \langle a \rangle_{u}&= (2\pi)^{-d} \int_{\mathbb{R}^{2d}} \nabla_z a(\Yeps\zeta + z) \exp\left(-\tfrac{1}{2}\zeta^\top \zeta\right)\, d\zeta\\ &= (2\pi)^{-d} \int_{\mathbb{R}^{2d}} (\nabla_\zeta a)(\Yeps\zeta + z) \exp\left(-\tfrac{1}{2}\zeta^\top\zeta\right)\, d\zeta = \langle\nabla_\zeta  a\rangle_{u}.
            \end{align*}
            Furthermore, again by integration by parts, 
            \begin{align*}
                \nabla_{\Yeps} \langle a \rangle_{u}&= (2\pi)^{-d} \int_{\mathbb{R}^{2d}} \nabla_{\Yeps} a(\Yeps\zeta + z) \exp\left(-\tfrac{1}{2}\zeta^\top \zeta\right)\, d\zeta\\ &= (2\pi)^{-d} \int_{\mathbb{R}^{2d}} (\nabla_\zeta a)(\Yeps\zeta + z) \zeta^\top  \exp\left(-\tfrac{1}{2}\zeta^\top \zeta\right)\, d\zeta \\ &= (2\pi)^{-d} \int_{\mathbb{R}^{2d}} (\nabla_\zeta a)(\Yeps\zeta + z)  \left(-\nabla_{\zeta}^\top  \exp\left(-\tfrac{1}{2}\zeta^\top \zeta\right) \right)\, d\zeta \\ &= (2\pi)^{-d} \int_{\mathbb{R}^{2d}} (\nabla_\zeta \nabla_\zeta^\top  a)(\Yeps\zeta + z)\Yeps   \exp\left(-\tfrac{1}{2}\zeta^\top \zeta\right)\, d\zeta
                = \langle \nabla_\zeta \nabla_\zeta^\top  a\rangle_u.
            \end{align*}
        \end{proof}

    %%%%%%%%%%%%%%%%%%%%%%%%%%%%%%%%%%%%%%%%%%%%%%%%%%%%%%%%%    
    \subsection{Proofs for the equations of motion}\label{sec:proofs_gen}

    \begin{proof}[Proof of \cref{theorem: symplectic evolution}]
    It was shown in \cite[Thm. 4.2]{BDHL2023} that for a general subquadratic Hamiltonian with symbol $h$ the evolution of the parameters satisfies
    \begin{align}\label{equation: equations of motion general}
        &\dot{q} = \langle \nabla_x h \rangle_u, \hspace{1cm}\dot{Q} = \langle \nabla_\xi \nabla_x^\top  h \rangle_u Q + \langle \nabla_\xi \nabla_\xi^\top  h \rangle_u P,\nonumber\\
        &\dot{p} = -\langle \nabla_\xi h \rangle_u, \hspace{.7cm}\dot{P} = -\langle \nabla_x \nabla_x^\top  h \rangle_u Q - \langle \nabla_x \nabla_\xi^\top  h \rangle_u P.
    \end{align}
    Applying the averaging identities of \cref{proposition: derivative average} to the Hamiltonian symbol $h$, we then see that the equations of motion \eqref{equation: equations of motion general} give \eqref{equation: equations of motion vectorized}.
    \end{proof}

    \begin{proof}[Proof of \cref{proposition: global well-posedness}]
Due to the smoothness of the classical symbol $h(t,\cdot)$ and the derivative identities of \cref{proposition: derivative average}, the variational Hamiltonian $\hB(t,\cdot)$ is locally Lipschitz continuous, and we have local well-posedness of the equations of 
motion~\eqref{equation: equations of motion vectorized}. Moreover, by 
\cite[Thm. 4.2]{BDHL2023}, the variational dynamics preserve symplecticity of the matrix $Y(t)$, guaranteeing a square integrable wave packet $u(t)$. Thus, for global well-posedness, it is enough to prove that the variational Hamiltonian inherits subquadratic growth from the classical symbol. 
If $h(t,\cdot)$ grows subquadratically, we consider $\mathbf k\in\mathbb N^{2D}$ with $\abs{\mathbf k}\ge 2$. We use the average formula \eqref{equation: weyl calculus} and obtain the estimate 
        \begin{align*}
            \abs{\partial_{\zB}^\mathbf{k}\hB(t,\zB)} &\le (2\pi)^{-d}\int_{\mathbb{R}^{2d}} \abs{\partial^\mathbf{k}h(t,\Yeps\zeta+z)}(1 + \norm{\zeta}_2)^\abs{\mathbf{k}}\exp\left(-\tfrac{1}{2}\zeta^\top\zeta\right)\,d\zeta \\
            &\le C_\mathbf{k}(t)C_{\mathbf{k},d},
        \end{align*}
        where $C_{\mathbf{k},d}$ bounds the $\abs{\mathbf{k}}$-th moment of the spherical Gaussian above.
    \end{proof}

\begin{proof}[Proof of \cref{corollary: equations of motion charged particle}]
   We obtain by direct calculation
    \begin{align*}
        \dot{\vB } &= \dot{\pB }-\frac{d}{dt}\AB  = (\nabla_{\qB}\AB ^\top )\pB -\nabla_{\qB}\VB -(\nabla_{\qB}\AB ^\top )^\top \dot{\qB }-\partial_t \AB  \\&= \left((\nabla_{\qB}\AB ^\top )-(\nabla_{\qB}\AB ^\top )^\top \right)\vB -\nabla_{\qB}\VB  + (\nabla_{\qB}\AB ^\top )\AB  -\partial_t \AB \\&=-\BB  \vB -\nabla_{\qB}\left(\VB -\frac{1}{2}\AB ^\top \AB \right) -\partial_t\AB.
    \end{align*}    
    
\end{proof}

    %%%%%%%%%%%%%%%%%%%%%%%%%%%%%%%%%%%%%%%%%%%%%%%%%%%%%%%%%%%%%%%%%
    \subsection{Proof for Boris numerics}\label{sec:proof_Boris}
    \begin{proof}[Proof of \cref{proposition: modified magnetic invariant}]
    We recall that with time-independent potentials and linear vector potential, the equations of motion read $\dot\qB = \vB$, 
    $\dot\vB = -\BB(\qB)\vB - \nabla_{\qB}\langle V\rangle_u$. We prove that Boris integrator corresponding to 
    \[
    \Psi_{\stepsize}^S = \Psi_{\frac{\stepsize }{2}}^\mathrm{kin}\circ \Psi_{\frac{\stepsize }{2}}^\mathrm{pot} \circ \Psi_{\stepsize}^\mathrm{mag} \circ \Psi_{\frac{\stepsize }{2}}^\mathrm{pot} \circ \Psi_{\frac{\stepsize }{2}}^\mathrm{kin}
    \] 
    leaves the modified structure matrix $\Omega_B(\tau)$ 
    invariant by first considering a kinetic half-step for $\dot\qB = \vB$, $\dot\vB = 0$. Its contribution to 
    $Y_1^\top\Omega_B(\tau)Y_1$ amounts to
\begin{align*}
            \begin{pmatrix}
                \Id & \frac{\stepsize }{2}\Id\\ 0 & \Id
            \end{pmatrix}^\top \begin{pmatrix}
                B &\Id\\-\Id & -\frac{\stepsize ^2}{4}B
            \end{pmatrix} \begin{pmatrix}
                \Id & \frac{\stepsize }{2}\Id\\ 0 & \Id
            \end{pmatrix} = \begin{pmatrix}
                B & \Id + \frac{\stepsize }{2}B \\ -\Id + \frac{\stepsize }{2}B & 0
            \end{pmatrix} =: \Omega_{\mathrm{kin}}. 
        \end{align*}
Then, we consider the potential part for $\dot\qB = 0$, $\dot\vB =- \nabla_{\qB}\langle V\rangle_u$, which reads on the 
level of the complex matrix factors $(Q,\Upsilon)$ as $\dot Q=0$, $\dot\Upsilon = -\langle \nabla^2V\rangle_uQ$. For the magnetic part, we have $\dot\qB = 0$, $\dot\vB = -\BB \vB$ and correspondingly $\dot Q = 0$, $\dot\Upsilon = -B\Upsilon$. We thus have to consider the triple matrix product
\begin{align*}
    \Phi_{\stepsize} &:= \begin{pmatrix}
        \Id & 0\\ -\frac\stepsize2\langle \nabla^2V\rangle_u & \Id
    \end{pmatrix}
    \begin{pmatrix}
        \Id & 0\\ 0 & R(\frac\stepsize2 B) 
    \end{pmatrix}
    \begin{pmatrix}
        \Id & 0\\ -\frac\stepsize2\langle \nabla^2V\rangle_u & \Id
    \end{pmatrix}\\
    &= \begin{pmatrix}
        \Id & 0\\ -\stepsize \left(\Id + \frac\stepsize2 B\right)^{-1}\langle\nabla^2V\rangle_u & R(\frac\stepsize2 B)
    \end{pmatrix},
\end{align*}
where the bottom left matrix block results from
\[
-\frac\stepsize2 \left(\Id+R\left(\frac\stepsize2 B\right)\right) \langle\nabla^2V\rangle_u = 
-\stepsize \left(\Id + \frac\stepsize2 B\right)^{-1}\langle\nabla^2V\rangle_u.
\]
We next calculate the action of this combined update on $\Omega_{\mathrm{kin}}$,
\[
\Phi_{\stepsize}^\top \Omega_{\mathrm{kin}} \Phi_{\stepsize} = \begin{pmatrix}
                B & \Id -\frac{\stepsize }{2}B\\ -\Id -\frac{\stepsize }{2}B & 0
            \end{pmatrix},
\]
using the skew-symmetry of $B$. For the concluding kinetic half-step, it then remains to observe 
        \begin{align*}
            \begin{pmatrix}
                \Id & \frac{\stepsize }{2}\Id\\ 0 & \Id
            \end{pmatrix}^\top \begin{pmatrix}
                B & \Id -\frac{\stepsize }{2}B\\ -\Id -\frac{\stepsize }{2}B & 0
            \end{pmatrix} \begin{pmatrix}
                \Id & \frac{\stepsize }{2}\Id\\ 0 & \Id
            \end{pmatrix} = \begin{pmatrix}
                B &\Id\\-\Id & -\frac{\stepsize ^2}{4}B
            \end{pmatrix}.
        \end{align*}
        \end{proof}

\subsection{Proof for the magnetic one-step map}\label{section: explicit implementation}
    
\begin{proof}[Proof of \cref{prop:stability_function}] We proceed in several steps so that the triangular form of the matrix $L$ 
allows us to bring in a Neumann sum. 

\paragraph{Step 1: Similarity transformation}
Let $S=\mathrm{diag}(b)\otimes I_D$ and note that $S$ and $M_n$ commute. Using
\[
S(\mathbf 1_s b^\top\otimes I_D)S^{-1}=(b\mathbf 1_s^\top\otimes I_D),\quad
S(\mathrm{diag}(b)^{-1}L^\top\mathrm{diag}(b)\otimes I_D)S^{-1}=(L^\top\otimes I_D),
\]
we obtain $W_n=S^{-1}\widetilde W_n S$,
where
\[
\widetilde W_n
=
I_{sD}
-\tau M_n\bigl((b\mathbf 1_s^\top-L^\top)\otimes I_D\bigr).
\]
Moreover,
$(b^\top\otimes I_D)S^{-1}=\mathbf 1_s^\top\otimes I_D$ and 
$S(\mathbf 1_s\otimes I_D)=b\otimes I_D$.
Hence
\begin{equation}\label{eq:Rreduced}
\hat R(\stepsize M_n)
=I_D+(\mathbf 1_s^\top\otimes I_D)\,\widetilde W_n^{-1}\,\stepsize M_n\,(b\otimes I_D)
\end{equation}
provided that $\widetilde W_n$ is invertible. 

\paragraph{Step 2: Triangular representation}
Using $(b\mathbf 1_s^\top)\otimes I_D=(b\otimes I_D)(\mathbf 1_s^\top\otimes I_D)$, we obtain the splitting
\[
\widetilde W_n
=
U_n-\stepsize M_n(b\otimes I_D)(\mathbf 1_s^\top\otimes I_D),
\qquad
U_n:=I_{sD}+\tau M_n(L^\top\otimes I_D).
\]
Consider now the linear system
$\widetilde W_n x = \stepsize M_n(b\otimes I_D) y$,
$y\in\mathbb R^D$,
that is, 
\[
U_n x=\stepsize M_n(b\otimes I_D)y+
\stepsize M_n(b\otimes I_D)(\mathbf 1_s^\top\otimes I_D)x.
\]
Assuming that $U_n$ is invertible, we left-multiply by $(\mathbf 1_s^\top\otimes I_D)U_n^{-1}$ to obtain
an equation for $z:=(\mathbf 1_s^\top\otimes I_D)x$, 
namely
\[
z =(\mathbf 1_s^\top\otimes I_D) U_n^{-1}\stepsize M_n(b\otimes I_D)y
+ (\mathbf 1_s^\top\otimes I_D)U_n^{-1}\stepsize M_n(b\otimes I_D)z.
\]
Provided the inverse exists, this reads
\[
z =\left(I_D-(\mathbf 1_s^\top\otimes I_D)U_n^{-1}\stepsize M_n(b\otimes I_D)\right)^{-1}
(\mathbf 1_s^\top\otimes I_D)U_n^{-1}\stepsize M_n(b\otimes I_D)y.
\]
Inserting into \eqref{eq:Rreduced} shows that
\begin{equation}\label{eq:RU}
\hat R(\stepsize M_n) =\left(I_D-(\mathbf 1_s^\top\otimes I_D)
U_n^{-1}\stepsize M_n(b\otimes I_D)\right)^{-1}.
\end{equation}

\paragraph{Step 3: Triangular splitting}
Since $L$ is strictly lower triangular, $L^\top$ is strictly upper triangular.
Because $M_n$ is block diagonal, 
$A_n := M_n(L^\top\otimes I_D)$ is strictly block upper triangular,
hence nilpotent of index at most $s$.
Therefore, the matrix $U_n$ defined in step 1 is invertible and its inverse is given by the Neumann sum
\[
U_n^{-1}=\sum_{k=0}^{s-1}(-\stepsize)^k A_n^k.
\]

\paragraph{Step 4: Explicit finite expansion}
We prove inductively, that for any $k\ge1$ and every block index $i=1,\ldots,s$,
\begin{equation}\label{eq:induc}
\bigl(A_n^{k-1}M_n(b\otimes I_D)\bigr)_i=\sum_{i=j_1<j_2<\cdots<j_k\le s}M_n^{(j_1)}\cdots M_n^{(j_k)}b_{j_k}
\prod_{\ell=1}^{k-1}L_{j_{\ell+1}j_\ell}.
\end{equation}
Then, substituting \eqref{eq:induc} into \eqref{eq:RU} proves our claim. For the inductive argument, we start with 
$k=1$ and observe that $(A_n^{0}M_n(b\otimes I_D))_i=b_i M_n^{(i)}$ for all $i$.
Then, we assume that \eqref{eq:induc} holds for some $k$ and write 
\begin{align*}
&(A_n^{k}M_n(b\otimes I_D))_i 
=\sum_{j=1}^s (A_n)_{ij}(A_n^{k-1}M_n(b\otimes I_D))_j\\
&=
\sum_{j=i+1}^s M_n^{(i)} L_{ji} \sum_{j=j_2<\cdots<j_{k+1}\le s} M_n^{(j_2)}\cdots M_n^{(j_{k+1})} b_{j_{k+1}}
\prod_{\ell=2}^{k}L_{j_{\ell+1}j_\ell},
\end{align*}
because $(A_n)_{ij}=M_n^{(i)} (L^\top)_{ij}=M_n^{(i)} L_{ji}$ and $(A_n)_{ij}=0$ unless $j>i$.
Re-indexing with $j_1=i$ then proves \eqref{eq:induc} for all $k\ge 1$.
\end{proof}

%%%%%%%%%%%%%%%%%%%%%%%%%%%%%%%
\subsection{Proofs for the symplectic integrators}\label{section: proofs magnetic integrator}

\begin{proof}[Proof \cref{theo:mag_disc}]
We proceed in three steps, addressing well-posedness, symplecticity, and order of the partitioned method for the magnetic evolution.

\paragraph{Well-posedness}
    We start by verifying that a step size restriction guarantees invertibility in \cref{prop:stability_function}. 
    We recall the definition of the magnetic parameter potential $\AB$ in \eqref{eq:pot_mag}, and note that the matrices $M_n^{(1)},\ldots,M_n^{(s)}$ are defined by evaluations of the transposed Jacobian 
    $J_{\AB}^\top = \nabla_{\qB} \AB^\top$ at different parameter positions, see \eqref{equation: RK rhs}. In general, these matrices do not commute.  
    We thus do not target a spectral estimate, but a norm bound. Arguing as in the proof of 
    \cref{proposition: global well-posedness}, we see that the derivative bounds up to order three on the vector potential 
    $A$ and a bound on the compact parameter set $K$ provide a uniform bound for the Jacobian $\nabla_{\qB}\AB$. Thus, we have a constant $C>0$ such that $\norm{M_n^{(i)}}_2 < C,\quad i\in \{1,\ldots,s\}$. We write the multi-variate polynomial  as $\rho(x) = \sum_{|\mathbf{k}|\le s} \rho_{\mathbf{k}} x^{\mathbf{k}}$, $x\in\mathbb R^s$, 
    and note that the coefficients $\rho_{\mathbf{k}}$ depend on the Runge--Kutta parameters $(L,b)$. Evaluating the polynomial on the matrices, we have
     \begin{align*}
        \norm{\rho(\stepsize M_n^{(1)}, \ldots, \stepsize M_n^{(s)})}_2&\le 
        \sum_{k=1}^s \rho_k\, C^k \stepsize^k
    \end{align*}
    with $\rho_k = \sum_{\mathbf{k}:|\mathbf{k}| = k} |\rho_\mathbf{k}|$. The uni-variate polynomial $\widetilde\rho(\stepsize) = \sum_{k=1}^s\rho_k \stepsize^k$ vanishes for $\stepsize=0$ and has nonnegative coefficients. Thus, there exists $\kappa_\rho>0$ such that $\widetilde\rho(\stepsize)<1$ for 
    $0<\stepsize<\kappa_\rho$. Setting $\stepsize_0 := \kappa_\rho/C$, we have
        \begin{align*}
            \norm{\rho(\stepsize M_n^{(1)}, \ldots, \stepsize M_n^{(s)})}_2<1
        \end{align*}
    and well-posedness of the one-step map $R(\stepsize M_n)$ for $\stepsize<\stepsize_0$.

\paragraph{Symplecticity} As already mentioned in Remark \ref{rem:quad}, the coefficients $(L,b)$, $(\hat L,b)$ satisfy the conditions for conserving quadratic invariants, which automatically renders the Runge--Kutta method symplectic, see e.g. \cite[\S VI.4.1]{HLW2006}.

\paragraph{Order} 
We consider the order conditions in terms of bi-colored rooted trees, using the notation of \cite[\S III.2]{HLW2006}. 
Since the method is symplectic, we can thus use the properties \cite[eq.~VI.(7.11) \& VI.(7.12)]{HLW2006}.
We observe, that the magnetic evolution \eqref{equation: magnetic hamiltonian evolution} is of the form 
$\dot{y} = f(y)$, $\dot{z} = g(y)z$. This structure implies that, for trees associated with a nonzero elementary differential, a white node can only have white children (as $f$ only depends on $y$), and a black node has at most one black child (as $g(y)z$ is linear in $z$), but an arbitrary number of white children. There are thus two types of trees to consider: mono-white trees and trees that consist of a chain of black nodes from a black root, where each such black node has an arbitrary number of white children.
\begin{figure}[H]
\centering
            \begin{tikzpicture}[
  scale=.7,
  level/.style={},
  every node/.style={circle, draw, fill=white, inner sep=1.5pt},
  filled/.style={circle, draw, fill=black, inner sep=1.5pt},
  edge/.style={thick},
  dottededge/.style={dotted, thick},
  subtree/.style={draw, rounded corners=16pt, thick}
]

\def\yroot{0}
\def\ylevelone{0.8}
\def\yleveltwo{1.6}
\def\ylevelthree{2.4}
\def\treeoffset{3.6}

\node[filled] (L0) at (0,\yroot) {};
\node[filled] (L1) at (0,\ylevelone) {};
\node[filled] (L2) at (0,\yleveltwo) {};

\node (L1a) at (0.6,\ylevelone) {};
\node (L2a) at (0.6,\yleveltwo) {};
\node (L3a) at (0,\ylevelthree) {};
\node (L3b) at (0.6,\ylevelthree) {};
\node (L1b) at (1.2,\ylevelone) {};
\node (L2b) at (1.5,\yleveltwo) {};

\draw[edge] (L0)--(L1);
\draw[edge] (L1)--(L2);
\draw[edge] (L0)--(L1a);
\draw[edge] (L0)--(L1b);
\draw[edge] (L1b)--(L2b);
\draw[edge] (L1)--(L2a);
\draw[edge] (L2)--(L3a);
\draw[edge] (L2)--(L3b);

\draw[->, thick] (1.9,1.4) -- (3.0,1.4);

\node (M0) at (0 + \treeoffset,\yroot) {};
\node[filled] (M1) at (0 + \treeoffset,\ylevelone) {};
\node[filled] (M2) at (0 + \treeoffset,\yleveltwo) {};

\node (M1a) at (0.6 + \treeoffset,\ylevelone) {};
\node (M2a) at (0.6 + \treeoffset,\yleveltwo) {};
\node (M3a) at (0 + \treeoffset,\ylevelthree) {};
\node (M3b) at (0.6 + \treeoffset,\ylevelthree) {};
\node (M1b) at (1.2 + \treeoffset,\ylevelone) {};
\node (M2b) at (1.5 + \treeoffset,\yleveltwo) {};

\draw[edge] (M0)--(M1);
\draw[edge] (M1)--(M2);
\draw[edge] (M0)--(M1a);
\draw[edge] (M0)--(M1b);
\draw[edge] (M1b)--(M2b);
\draw[edge] (M1)--(M2a);
\draw[edge] (M2)--(M3a);
\draw[edge] (M2)--(M3b);

\draw[thick, rotate around={75:(3.8, 1.75)}]
  (3.8, 1.75) ellipse (1.1 and 0.55);
\node[draw = none, fill = none] at (4.5,2.6) {$v$};

\draw[thick, rotate around={48:(4.4, 0.74)}]
  (4.4, 0.74) ellipse (1.36 and 0.38);
\node[draw = none, fill = none] at (5.2,0.7) {$u$};

\node[draw = none, fill = none] at (\treeoffset,-0.6) {$u\circ v$};

\draw[->, thick] (1.9 + \treeoffset,1.4) -- (3.0 + \treeoffset,1.4);

\node (R0) at (1.2 + 2*\treeoffset,\ylevelone) {};
\node[filled] (R1) at (0 + 2*\treeoffset,\yroot) {};
\node[filled] (R2) at (0 + 2*\treeoffset,\ylevelone) {};

\node (R1a) at (1.2 + 2*\treeoffset,\yleveltwo) {};
\node (R2a) at (0.6 + 2*\treeoffset,\ylevelone) {};
\node (R3a) at (0 + 2*\treeoffset,\yleveltwo) {};
\node (R3b) at (0.6 + 2*\treeoffset,\yleveltwo) {};
\node (R1b) at (1.8 + 2*\treeoffset,\yleveltwo) {};
\node (R2b) at (1.8 + 2*\treeoffset,\ylevelthree) {};

\draw[edge] (R0)--(R1);
\draw[edge] (R1)--(R2);
\draw[edge] (R0)--(R1a);
\draw[edge] (R0)--(R1b);
\draw[edge] (R1b)--(R2b);
\draw[edge] (R1)--(R2a);
\draw[edge] (R2)--(R3a);
\draw[edge] (R2)--(R3b);

\draw[thick, rotate around={90:(0.2 + 2*\treeoffset,0.1 + \ylevelone)}]
  (0.2 + 2*\treeoffset,0.1 +\ylevelone) ellipse (1.2 and 0.7);
\node[draw = none, fill = none] at (7.5,2.25) {$v$};

\draw[thick, rotate around={80:(1.5 + 2*\treeoffset,0.8 + \ylevelone)}]
  (1.5 + 2*\treeoffset,0.8 +\ylevelone) ellipse (1.2 and 0.5);
\node[draw = none, fill = none] at (9.4,2.25) {$u$};

\node[draw = none, fill = none] at (2*\treeoffset,-0.6) {$v\circ u$};

\end{tikzpicture}
\caption{Exemplary visualization of the induction step.}
\end{figure}
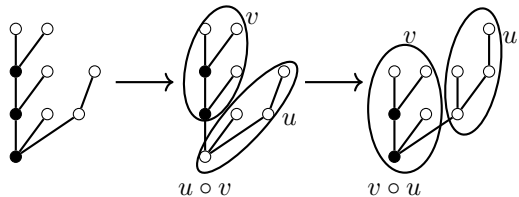

Now, consider such rooted trees with at most $\nu$ nodes; we do an induction on the number of black nodes. 
 If the tree is mono-white, the associated order condition is satisfied by the order of the explicit integrator. 
Assume that the order condition holds for trees with $n$ black nodes of the above type, by \cite[eq. VI.(7.12)]{HLW2006}, we can change the color of the root to white. The tree now can be written as $u\circ v$ with $u$ (sub-tree of the root and all its white children) mono-white and $v$ (sub-tree of the unique black child of the root); by induction, the order conditions are satisfied for $u, v$ and $v\circ u$, so by \cite[eq. VI.(7.11)]{HLW2006} they are also satisfied for $u\circ v$.
\end{proof}
    
\begin{proof}[Proof of \cref{theorem: symplectic integrator}]
  We prove that the splitting integrator conserves the quadratic invariants $\qB^\top I_k\pB$ derived in \eqref{equation: quadratic}. We consider the last $2d^2$ components of $\qB$ and $\pB$ each and collect them as a matrix $Y\in\mathbb R^{2d\times 2d}$. We visit the three sub-steps of the splitting scheme. The  $Y$-updates given by the (time-averaged) kinetic or potential Hamiltonians ($\dot{\qB } = \pB$, $\dot{\pB } = 0$ or $\dot{\qB } = 0$, $\dot{\pB } = -\nabla_{\qB} \VB(\qB)$) are of the form
            \begin{align*}
                Y^\mathrm{new} = \begin{pmatrix}
                    \Id & \stepsize \Id \\0 & \Id
                \end{pmatrix}Y,\  \text{ or } Y^\mathrm{new} = \begin{pmatrix}
                    \Id & 0 \\ -\stepsize\nabla^2_{q}V(q) & \Id
                \end{pmatrix}Y.
            \end{align*}
            Both updates leave $\Omega$ invariant, and thus conserve \eqref{equation: symplecticity} respectively \eqref{equation: quadratic}. By construction of the partitioned Runge--Kutta method, the magnetic sub-step conserves the quadratic invariant \eqref{equation: quadratic}, too, see also Remark  \ref{rem:quad}. Similar arguments holds for the linear (using \cite[Thm. IV.1.5]{HLW2006}) and the semiclassical angular momentum, 
            if the potentials have the respective symmetries.
\end{proof}

    %%%%%%%%%%%%%%%%%%%%%%%%%%%%%%%%%%%%%%%%%%%
    \subsection{Proofs for energy conservation}\label{sec:prooferr}
    For establishing \Cref{theorem: near-cons energy}, we prove that the average $\hB(\zB)$ is bounded and analytic if the parameters are bounded. 
    
    \begin{lemma}[Boundedness]\label{lem:bound} Let $h$ be of at most exponential growth, that is, there exist $c_1, c_2 > 0$ such that $\abs{h(w)}\le c_1\exp(c_2\norm{w}_2)$ for all $w\in\mathbb C^d$. Then, 
    \[
    \abs{\hB(\zB)}  \le c_1 C_d \exp\!\left(\frac{c_2^2}{2} \norm{\Yeps}_2^2 + c_2\norm{z}_2\right) \left( c_2^{2d-1}\norm{\Yeps}_2^{2d-1}+ 1\right)
    \]
    for all $\zB = (z,\Yeps)\in\mathbb C^{2D}$, where $C_d>0$ depends on the dimension $d$ only.
    \end{lemma}

    \begin{proof}
    By the exponential bound and the integral representation \eqref{equation: weyl calculus}, we get
            \begin{align*}
                \abs{\hB(\zB)} &\le
                (2\pi)^{-d} \int_{\mathbb R^{2d}} 
                |h(\Yeps\zeta + z)| \exp(-\tfrac12\zeta^\top\zeta) \, d\zeta\\
                &\le c_1 \int_{\mathbb{R}^{2d}} \exp\left( -\tfrac{1}{2}\norm{\zeta}_2^2 + c_2\left(\|\Yeps\|_2 \norm{\zeta}_2 + \norm{z}_2\right)\right)\, d\zeta
                %\\& \le C(c_1, c_2, d)\exp\left(c_2^2 \norm{\Yeps}_2^2 + c_2\norm{z}_2\right).
            \end{align*}
    For $a = c_2 \norm{\Yeps}_2>0$, we have 
    \begin{align*}
        & \int_{\mathbb{R}^{2d}} \exp\left( -\tfrac{1}{2}\norm{\zeta}_2^2 + a \norm{\zeta}_2 \right) d\zeta 
             = \frac{2\pi^d}{\Gamma(d)} \int_0^\infty r^{2d-1} e^{-r^2/2 + ar} dr\\
             &\le \frac{2\pi^d}{\Gamma(d)} \, e^{a^2/2}\left(  (2a)^{2d-1} \int_{0}^{2a} e^{-(r-a)^2/2} dr + \int_{2a}^\infty r^{2d-1} e^{-r^2/8}dr\right) \\
             &\le \frac{2\pi^d}{\Gamma(d)} \, e^{a^2/2}\left(  (2a)^{2d-1} \sqrt{2\pi} + \int_{0}^\infty r^{2d-1} e^{-r^2/8}dr\right) \\
             &\le C_d \, e^{a^2/2} \left( a^{2d-1} + 1\right) .
    \end{align*}
    \end{proof}
      
    \begin{lemma}[Analyticity]
        If $h(z)$ is an entire function of at most exponential growth, then  
        the averaged Hamiltonian $\hB(\zB)$ is an analytic function on each poly-disc 
        $\left\{\zB \in \mathbb{C}^{2D}\mid \norm{z}_2\le R,\, \norm{\Yeps}_2\le M\right\}$ and bounded by a constant, that depends 
        on $R,M$ and the constants $C_d,c_1,c_2$ of \Cref{lem:bound}. 
        \begin{proof}
            We use Morera's theorem for each coordinate $\zB_j$, $j=1,\ldots,2D$, which is enough by Hartog's theorem \cite[Thm. 2.2.8]{Ho73}, to show analyticity, i.e., we need to prove \begin{align*}
                \oint_C \hB(\zB)\, d\zB_j = 0
            \end{align*}
            for closed piecewise $C^1$ curves in $\mathbb{C}$. Since $\hB(\zB)$ is given by the phase space integral \eqref{equation: weyl calculus}, we need to justify exchanging the order of integration. By \Cref{lem:bound} and Fubini, we can interchange, and since $h(\Yeps\zeta + z)$ is analytic in $z, \Yeps$, we can again employ Morera's theorem to conclude that the contour integral is zero.
        \end{proof}
    \end{lemma}
    
%%%%%%%%%%%%%%%%%%%%%%%%%%%%%%%%%%%%%%%%%%%%%%%%%    
\section{Numerical experiments}\label{sec:num}

 We present numerical examples to showcase the structure-preservation and long-time behavior of the proposed symplectic integrator. We use the second-order version specified below in \cref{sec:second}. An order four integrator can be constructed using, e.g., the splitting integrator of \cite[eq.(63)]{BM2001} and the classic Runge--Kutta order four scheme. 
 
\subsection{Symplectic splitting integrator of order two}\label{sec:second}
To compare the symplectic splitting and the second order Boris-type method \eqref{equation: Boris one-step map}, we consider the order two  integrator based on a mid-point time-average, Strang splitting, and the partitioned Runge--Kutta method given by Heun's rule:
    \begin{align}\label{equation: order 2 example}\renewcommand\arraystretch{1.2}
        \begin{array}{c|cc}
        0\\
        1 & 1\\
        \hline
        & \frac{1}{2} &\frac{1}{2}
        \end{array},\qquad
        \renewcommand\arraystretch{1.2}
        \begin{array}
        {c|cc}
        0&\frac{1}{2}&-\frac{1}{2}\\ 
        1 & \frac{1}{2}&\frac{1}{2}\\
        \hline
        & \frac{1}{2} &\frac{1}{2}
        \end{array}.
    \end{align}
    The position update is $\qB_{n+1} = \qB_n - \frac{\stepsize}{2}\left(\AB(\qB_n) + \AB(\qB_n^-)\right)$ with $\qB_n^- = \qB_n-\stepsize \AB(\qB_n)$, and by \cref{prop:stability_function} the momentum update can be written as
    \begin{align*}
        \vB_{n+1} = \bigg(\Id_{D} & - \frac{\stepsize}{2}\left(J_{\AB}^\top(\qB_n) + J_{\AB}^\top(\qB_n^-\right) 
        %\\ & 
        + \frac{\stepsize^2}{2}J_{\AB}^\top(\qB_n) J_{\AB}^\top(\qB_n^-)\bigg)^{-1}\vB_n,
    \end{align*}
    where $\AB$ and $J_{\AB}^\top = \nabla_{\qB} \AB ^\top$ are evaluated at time $t_n+\frac{\stepsize}{2}$, if the vector potential $A$ depends explicitly on time. For this method, $\kappa_\rho = \sqrt{3}-1\approx 0.73$ in the proof of well-posedness in \cref{theo:mag_disc}.
    We provide a general-purpose implementation of this integrator\footnote{\url{https://gitlab.lrz.de/00000000014AA221/magn_wave_packet_integration}} based on 
    tensorized Gauss--Hermite quadrature with $N^d$ nodes ($N = 7,5,11$ in subsections \ref{sec:sin}, \ref{sec:Penning}, \ref{sec:sym}, respectively) and automatic differentiation, which was used for the following numerical examples.
\subsection{Two-dimensional nonlinear vector potential}\label{sec:sin}
  First, we consider a two-dimensional, nonlinear, trigonometric vector potential $A$ similar to \cite[\S 6.1]{SBHL2025} in combination with a quadratic scalar potential $V$ to confine the trajectory, 
    \begin{align}\label{equation: non-lin ex}
        A(t, x) = \begin{pmatrix}
            \sin(x_1 + x_2 + \alpha \sin(t))\\ -\sin(x_1 + x_2 + \alpha \sin(t))
        \end{pmatrix}, \quad V(t,x) = x_1^2 + x_2^2.
    \end{align}
    For illustration of structure preservation, we set $\alpha = \frac{1}{2}$; for near-conservation of energy, we set $\alpha = 0$.
    We set the semiclassical parameter $\varepsilon = 0.001$. 
    As initial condition, we pick 
    \begin{align*}
        q_0 = \begin{pmatrix}
            1\\1
        \end{pmatrix},\,p_0 = \begin{pmatrix}
            1\\0
        \end{pmatrix},\,Q_0 = \Id_2,\,P_0 = i\Id_2,\, S_0 = 0 \text{ and } t_0 = 0.
    \end{align*} 
    \begin{figure}[H]
        \centering
        \includegraphics[width=0.48\linewidth]{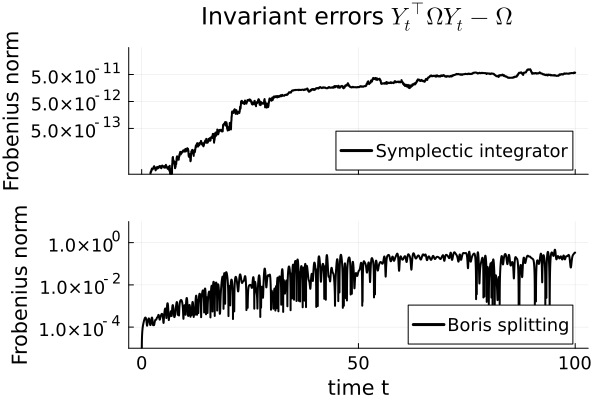}
        \quad
        \includegraphics[width=0.48\linewidth]{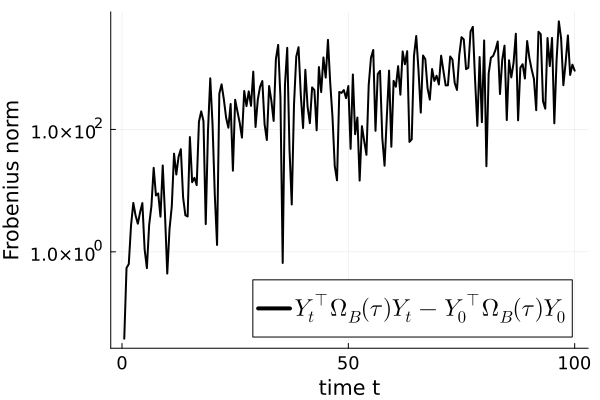}
        \caption{The symplectic and the Boris splitting integrator are applied to system \eqref{equation: non-lin ex} with $\alpha = 1/2$ for step size $\tau = 0.01$; the plots on the left show the deviation from symplecticity over time, the plot on the right illustrates that a nonlinear vector potential destroys the invariance proven in \cref{proposition: modified magnetic invariant}; the $y$-axis is scaled logarithmically.\label{fig:nonlinear_symp}}
    \end{figure}
    In \Cref{fig:nonlinear_symp}, we illustrate structure preservation for the two integrators. Generally, we would expect the error of the quadratic Hagedorn invariant $Y_t^\top\Omega Y_t$ from \eqref{equation: symplecticity} to be around $\varepsilon^{-\frac{1}{2}}$ times machine precision, as all calculations are performed on the scale of $\Yeps = \sqrt{\frac{\varepsilon}{2}}\, Y$, and we have an error of order machine precision there. Here, due to the complicated dynamics,  $\Yeps$ grows to order 1, and we thus see an error of magnitude $\varepsilon^{-1}$ times machine precision.
    Similarly, the symplecticity error for the Boris splitting integrator is of the order of the step size squared times $\varepsilon^{-1}$,  $\stepsize^2\varepsilon^{-1}=10^{-1}$, with oscillations and a slight drift. Since the vector potential $A$ is nonlinear, the modified invariance property of \Cref{proposition: modified magnetic invariant} cannot be expected to hold, as illustrated in the plot on the right. In \Cref{fig:nonlinear_ener}, we see both long-time near-conservation and the second order of the error with respect to the step size $\tau$, for the symplectic integrator, as predicted by \cref{theorem: near-cons energy}. The Boris splitting integrator, in contrast, exhibits a slight energy drift despite the presence of a quadratic confining potential $V$. It seems that the near-conservation results for the classical Boris integrator \cite[\S 4]{HL2018} do not extend to the semiclassical case.
    \begin{figure}[H]
        \centering
        \includegraphics[width=0.4\linewidth]{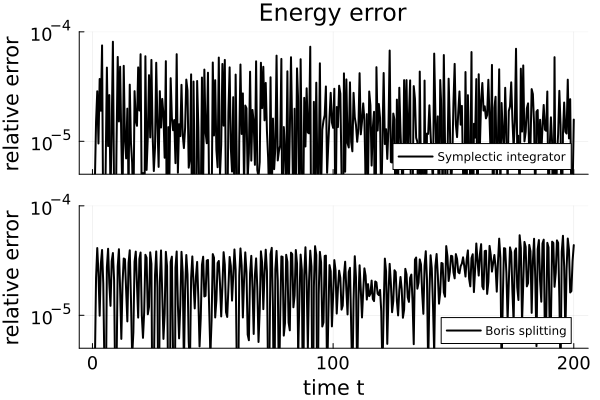}
        \quad
        \includegraphics[width=0.4\linewidth]{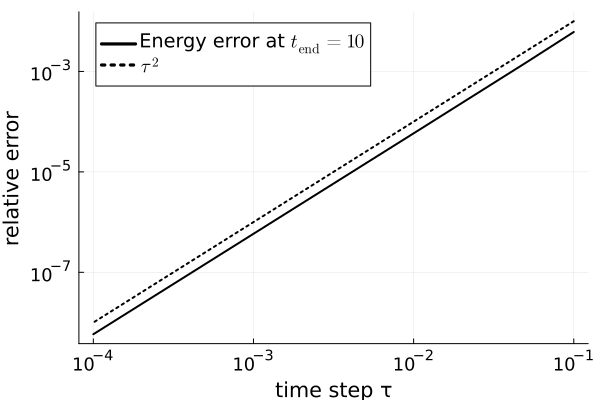}
        \caption{
        The symplectic and the Boris splitting integrator are applied to system \eqref{equation: non-lin ex} with $\alpha = 0$ for step size $\tau = 0.01$; on the left, we plot the relative energy error of both integrators over time; the $y$-axis is scaled logarithmically. On the right, we plot the relative error of the energy of the symplectic integrator at time $10$ for different step sizes against a dashed reference line $\stepsize\mapsto\stepsize^2$; both axes are scaled logarithmically.\label{fig:nonlinear_ener}}
    \end{figure}

\subsection{Three-dimensional Penning trap}\label{sec:Penning}
    We explore a system with linear vector potential. The three-dimensional model concerns a charged microscopic particle (an electron or a proton) in a macroscopic hyperbolic Penning trap as derived in \cite[\S 2]{SBHL2025}. We have $A$ and $V$ given by
    \begin{align}\label{equation: Penning trap}
        A(t, x) = 57.125\begin{pmatrix}
            -x_2\\x_1\\0
        \end{pmatrix}, \quad V(t,x) = 113.25\left(x_3^2 - \frac{1}{2}\left(x_1^2+x_2^2\right)\right),
    \end{align}
     semiclassical parameter $\varepsilon = 1.19\cdot10^{-8}$ and initial conditions 
     \begin{align*}
        q_0 = \begin{pmatrix}
            0.133\\0.133\\0.258
        \end{pmatrix},\,p_0 = \begin{pmatrix}
            0.133\\7.492\\3.879
        \end{pmatrix},\,Q_0 = \mathrm{diag}(q_0),\,P_0 = iQ_0^{-1},\, S_0 = 1.009 \text{ and } t_0 = 0.
    \end{align*}
    \Cref{fig:penning} shows, that the Hagedorn invariant oscillates but remains bounded for the Boris splitting. The modification of the invariant for the Boris splitting formulation has an error of magnitude $\varepsilon^{-\frac{1}{2}}$ times machine precision, very similar to the error of the symplectic invariant of the symplectic integrator. This shows that \cref{proposition: modified magnetic invariant} correctly predicts the behavior of the symplectic invariant up to machine errors.
    \begin{figure}[H]
        \centering
        \includegraphics[width=0.4\linewidth]{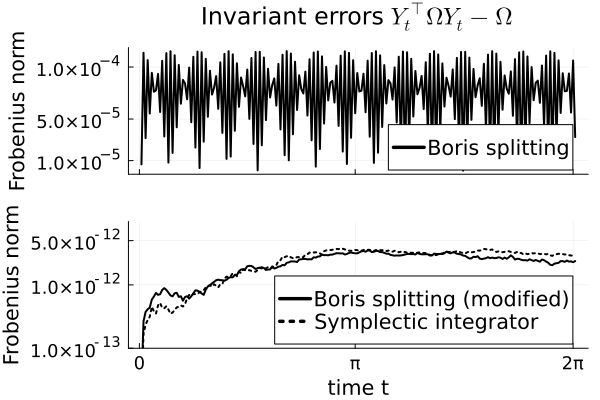}
        \includegraphics[width=0.4\linewidth]{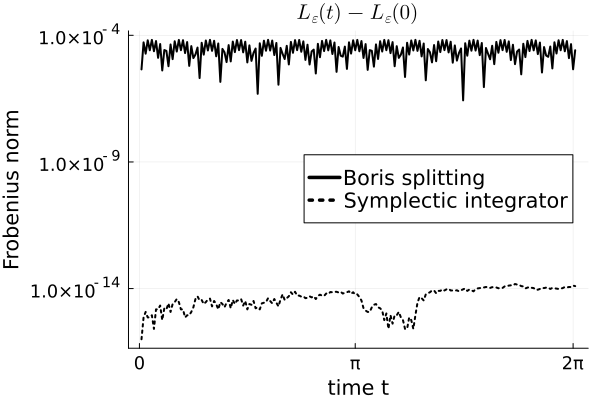}
        \caption{The Boris and the symplectic splitting integrator is applied to the Penning trap with step size $\stepsize = 0.001$. The upper left plot shows the Boris error in the invariant $Y_t^\top\Omega Y_t - \Omega$. Below, we compare the modified invariant  $Y_t^\top \Omega_B(\stepsize) Y_t - Y_0^\top\Omega_B(\stepsize)Y_0$ with the symplectic invariant of the symplectic integrator. On the right, we compare the semiclassical angular-momentum error in the $ x_1$- $ x_2$ plane for both integrators. In both plots, the $y$-axis is scaled logarithmically.}\label{fig:penning}
    \end{figure}

    \subsection{Two-dimensional symmetric vector potentials}\label{sec:sym}
    Finally, we explore the conservation of linear and angular momenta; we modify the potential \eqref{equation: non-lin ex} to have the required symmetries, but keep the initial conditions and the semiclassical parameter $\varepsilon = 0.001$. We use 
    \begin{align*}
        A(t, x) = \begin{pmatrix}
            \sin(x_1 - x_2)\\ \sin(x_1 - x_2)
        \end{pmatrix}, \quad V(t,x) = (x_1-x_2)^2
    \end{align*}
    for translational and 
    \begin{align*}
        A(t, x) = \frac{1}{1+x_1^2 + x_2^2} \begin{pmatrix}
            -x_2\\ x_1
        \end{pmatrix}, \quad V(t,x) = \frac{1}{2}(x_1^2 + x_2^2)
    \end{align*}
    for rotational symmetry. Note that rotational symmetry means that $A$ is covariant under rotation, i.e., $A(Rx) = RA(x)$ for all $R\in SO(2)$ and $x\in\mathbb R^2$, which is satisfied by the above potential. As discussed in \cref{sec:invariants}, we monitor the total linear momentum $p_1(t) + p_2(t) = (1,1)^\top p(t)$ and the semiclassical angular momentum $L_\varepsilon(t)$ for the translational and the rotational symmetric case, respectively. As expected by \Cref{cor:conserve}, the plots in \cref{fig:sym} show conservation at roughly machine precision. 

    \begin{figure}[H]
        \centering
        \includegraphics[width=0.48\linewidth]{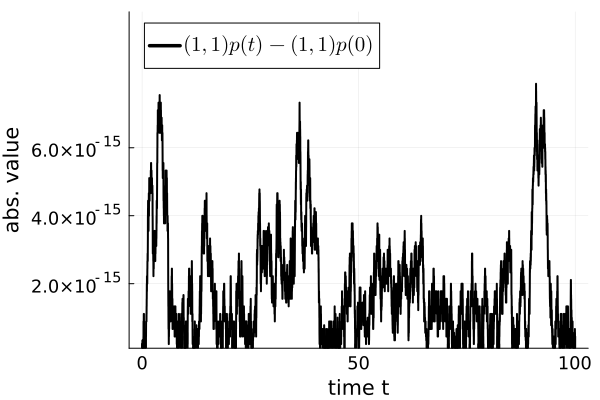}
        \quad
        \includegraphics[width=0.48\linewidth]{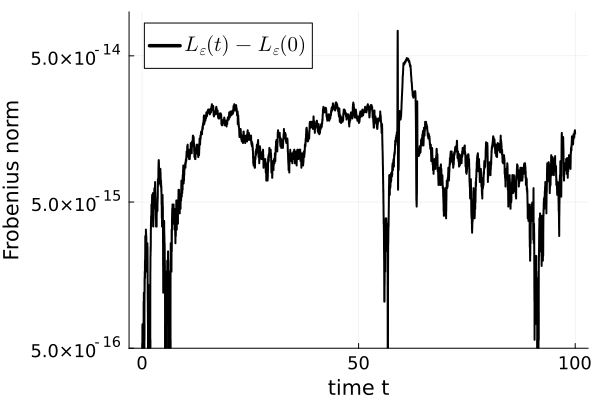}
        \caption{We apply the symplectic integrator with step size $\stepsize = 0.01$ to the two symmetric systems described above. On the left, we plot the error of the total linear momentum for the first system; on the right, we plot the error of the semiclassical angular momentum for the second system, here the $y$-axis is scaled logarithmically. \label{fig:sym}}
    \end{figure}

%%%%%%%%%%%%%%%%%%%%%%%%%%%%%%%%%%%%
\section{Conclusion}
We have developed a class of structure-preserving time integration schemes for Gaussian wave packet dynamics associated with the magnetic Schrödinger equation. We constructed Boris-type methods as well as explicit high-order symplectic integrators based on splitting and partitioned Runge–Kutta techniques. While the Boris approach captures key features of the magnetic flow, it does not exactly preserve the quadratic invariants underlying the Hagedorn parametrization. In contrast, the proposed symplectic schemes conserve these invariants, thereby guaranteeing square integrability of the wave packet and ensuring consistency of the variational approximation over long time intervals. The analysis established uniform error bounds in the semiclassical parameter for both the wave packet parameters and observable quantities, together with near-conservation of the averaged Hamiltonian over exponentially long times. Numerical experiments confirm the favorable long-time and structure preserving behavior of the integrators.
 
\appendix
\section{Wigner--Weyl properties}

\begin{lemma}[Gaussian Wigner function]\label{proposition: wigner weyl calculus}
        The Wigner function $\mathcal W_u$ of a normalised Gaussian wave packet in Hagedorn parametrization $u = u[q,p,Q,P,S]$ is a phase space Gaussian
        \begin{align*}
            \mathcal{W}_u(\zeta) = \frac{(2\pi)^{-d}}{\mathrm{det}(\Sigma_\varepsilon)}\exp\left(-\frac{1}{2}(\zeta-z)^\top\Sigma_\varepsilon^{-1}(\zeta-z)\right),\quad \zeta\in\mathbb R^{2d},
        \end{align*}
        with mean and covariance matrix
        \begin{align*}
            z = 
            \begin{pmatrix}
                q\\ p
            \end{pmatrix},\quad
            \Sigma_\varepsilon = \frac{\varepsilon}{2}
            \begin{pmatrix}
                QQ^* & \Re(QP^*)\\ \Re(PQ^*) & PP^*
            \end{pmatrix}.
        \end{align*}
        Moreover, we can rewrite $\Sigma_\varepsilon = \frac{\varepsilon}{2}YY^\top$, with $Y$ from \eqref{equation: symplecticity}.
        \end{lemma}
        
        \begin{proof}
            It was shown in \cite[Prop. 6.15 \& Lem. 6.17]{LL2020} that $\mathcal{W}_u$ is a Gaussian
            \begin{align*}
                \mathcal{W}_u = (\varepsilon\pi)^{-d}\exp\left(-\frac{1}{\varepsilon}(\zeta-z)^\top G (\zeta-z)\right),\text{ with } G = \begin{pmatrix}
                PP^* & -\Re(PQ^*)\\ -\Re(QP^*) & QQ^*
            \end{pmatrix},
            \end{align*}
            with $G$ symplectic. Now, observe that
            \begin{align*}
                \Omega YY^\top\Omega^\top &= \begin{pmatrix}
                0 & \Id \\ -\Id & 0
            \end{pmatrix} \begin{pmatrix}
                QQ^* & \Re(QP^*)\\ \Re(PQ^*) & PP^*
            \end{pmatrix} \begin{pmatrix}
                0 & -\Id \\ \Id & 0
            \end{pmatrix} \\&= 
            \begin{pmatrix}
                PP^* & -\Re(PQ^*)\\ -\Re(QP^*) & QQ^*
            \end{pmatrix} = G
            \end{align*}
            and thus $G = \frac{\varepsilon}{2}\Sigma_\varepsilon^{-1}$ using $G^{-1} = \Omega G\Omega^\top$ as $G$ is symplectic.
        \end{proof}

\begin{lemma}[Hamiltonians]\label{lem:symbol}
The magnetic Schr\"odinger operator given in \eqref{eq:mag_ham} is the semiclassical Weyl-quantization $\mathcal H(t) = \mathrm{op}(h(t))$ of the classical symbol
\[
h(t,x,\xi) = \frac{1}{2}\abs{\xi}^2 - A(t,x)^\top \xi + \frac{1}{2}\abs{A(t,x)}^2 + V(t,x)
\]
with $(t,x,\xi)\in\mathbb R\times\mathbb R^d\times \mathbb R^d$. For a normalised Gaussian $u\in\mathcal M$ with canonical parameters 
$(\qB,\pB)\in\mathbb R^D\times\mathbb R^D$ the average $\hB (t,\qB ,\pB ):=\langle h(t,\cdot)\rangle_u$ satisfies
\[
\hB (t,\qB ,\pB )=\frac{1}{2}\pB ^\top \pB -\pB ^\top \AB (t,\qB )+\VB (t,\qB )
\]
with the averaged potentials $\AB(t,\qB)$ and $\VB(t,\qB)$ defined \cref{sec:mag_dyn}. If $A(t,x) = M_A(t)x$ for some $M_A(t)\in\mathbb R^{d\times d}$, then 
\[
\VB (t,\qB ) = \tfrac12\AB(t,\qB)^\top \AB(t,\qB) + \langle V(t,\cdot)\rangle_u. 
\]
\end{lemma}

\begin{proof}
For notational simplicity, we suppress time-dependence. We start by calculating the Weyl-symbol and expanding the noncommutative square,
\[
\abs{-i\varepsilon\nabla-A(x)}^2 = 
\abs{-i\varepsilon\nabla}^2 - \left( (-i\varepsilon\nabla)^\top A(x) + A(x)^\top (-i\varepsilon\nabla)\right)
+ \abs{A(x,x)}^2.
\]
We use $-i\varepsilon\nabla = \mathrm{op}(\xi)$ and rewrite the mixed terms using Weyl calculus, 
see \cite[Theorem 2.7.4]{Martinez2002}. We obtain
\begin{align*}
(-i\varepsilon\nabla)^\top A(x) 
&=  \mathrm{op}(\xi^\top A(x) + \tfrac{\varepsilon}{2i} \{\xi,A(x)\})
= \mathrm{op}(A(x)^\top\xi + \tfrac{\varepsilon}{2i} \mathrm{div}A(x)),\\
A(x)^\top (-i\varepsilon\nabla) &= \mathrm{op}(A(x)^\top\xi + \tfrac{\varepsilon}{2i}\{A(x),\xi\})
= \mathrm{op}(A(x)^\top\xi - \tfrac{\varepsilon}{2i} \mathrm{div}A(x)).
\end{align*}
The terms with $\mathrm{div}A(x)$ have opposite signs and cancel, and we arrive at the claimed formula. 
For the average, we use the Gaussian Wigner function of \cref{proposition: wigner weyl calculus} and standard formulas for 
its second moments to obtain
$\langle|\xi|^2\rangle_u = \pB^\top\pB$.
For the contribution, that is mixed in position and momentum, we write 
\[
\langle \xi^\top A(x)\rangle_u = p^\top \langle A(x)\rangle_u + \langle (\xi-p)^\top A(x)\rangle_u
\]
and use the Gaussian's gradient $\partial_\zeta \mathcal W_u(\zeta) 
=- \Sigma_\varepsilon^{-1}(\zeta-z) \mathcal W_u(\zeta)$. We write the covariance matrix in block decomposition
$\Sigma_\varepsilon = \left(\Sigma_{11},\Sigma_{12}; \Sigma_{21},\Sigma_{22}\right)$,
and perform an integration by parts, 
\begin{align*}
\langle (\xi-p)^\top A(x)\rangle_u &= 
\int_{\mathbb R^{2d}}  (\xi-p)^\top A(x) \mathcal W_u(\zeta) d\zeta\\
&= - \sum_{k=1}^d   \int_{\mathbb R^{2d}} A_k(x) \left(\Sigma_{21}\partial_x + \Sigma_{22}\partial_\xi\right)_k 
\mathcal W_u(\zeta) d\zeta\\
&= \sum_{k,\ell=1}^d   \int_{\mathbb R^{2d}} \partial_\ell A_k(x) (\Sigma_{21})_{k\ell}  
\mathcal W_u(\zeta) d\zeta\\
&= \frac\varepsilon2\sum_{k,\ell=1}^d \langle \partial_\ell A_k\rangle_u (\Re(P)\Re(Q)+\Im(P)\Im(Q))_{k\ell}, 
\end{align*}
which proves that $\langle\xi^\top A(x)\rangle_u = \pB^\top \AB(\qB)$. The remaining terms then define $\VB(\qB)$. 
In the linear case, the squared vector potential generates second moments, and we have
    \[
    \langle |A(t,\cdot)|^2\rangle_u = \langle x^\top M_A(t)^\top M_A(t)x\rangle_u = \AB(t,\qB)^\top \AB(t,\qB).
    \]
\end{proof}

\section*{Acknowledgments}
The authors would like to thank the organizers of the Winter School on "Mathematical Challenges in Quantum Mechanics" at Gran Sasso Science Institute, where the first discussions leading to this work took place. Furthermore, the first author would like to thank Carl Quitter for many helpful discussions. 

We have used Grammarly for editing and polishing written text. Furthermore, we used Claude Code as a coding assistant for the numerical experiments.

\newpage
\bibliographystyle{siamplain}
\bibliography{references}
\end{document}

\maketitle

\section{A detailed example}

Here we include some equations and theorem-like environments to show
how these are labeled in a supplement and can be referenced from the
main text.
Consider the following equation:
\begin{equation}
  \label{eq:suppa}
  a^2 + b^2 = c^2.
\end{equation}
You can also reference equations such as \cref{eq:matrices,eq:bb} 
from the main article in this supplement.

\lipsum[100-101]

\begin{theorem}
An example theorem.
\end{theorem}

\lipsum[102]
 
\begin{lemma}
An example lemma.
\end{lemma}

\lipsum[103-105]

Here is an example citation: \cite{KoMa14}.

\section[Proof of Thm]{Proof of \cref{thm:bigthm}}
\label{sec:proof}

\lipsum[106-112]

\section{Additional experimental results}
\Cref{tab:smfoo} shows additional
supporting evidence. 

\begin{table}[htbp]
\footnotesize
  \caption{Example table.}\label{tab:smfoo}
\begin{center}
  \begin{tabular}{|c|c|c|} \hline
   Species & \bf Mean & \bf Std.~Dev. \\ \hline
    1 & 3.4 & 1.2 \\
    2 & 5.4 & 0.6 \\ \hline
  \end{tabular}
\end{center}
\end{table}

\bibliographystyle{siamplain}
\bibliography{references}

%% file: shared.tex
\usepackage{lipsum}
\usepackage{amsfonts}
\usepackage{graphicx}
\usepackage{epstopdf}
\usepackage{algorithmic}
\usepackage{amssymb,mathtools}
\usepackage{float}
\usepackage{dsfont}
\usepackage{enumitem}
\usepackage{physics}

\usepackage{tikz}
\usetikzlibrary{fit}
\usetikzlibrary{arrows.meta, positioning, fit}

\ifpdf
  \DeclareGraphicsExtensions{.eps,.pdf,.png,.jpg}
\else
  \DeclareGraphicsExtensions{.eps}
\fi

\newsiamremark{remark}{Remark}
\newsiamremark{hypothesis}{Hypothesis}
\newsiamthm{assumption}{Assumption}
\newsiamremark{notation}{Notation}
\crefname{hypothesis}{Hypothesis}{Hypotheses}
\crefname{assumption}{Assumption}{Assumptions}
\crefname{notation}{Notation}{Notations}
\newsiamthm{claim}{Claim}
\newsiamremark{fact}{Fact}
\crefname{fact}{Fact}{Facts}

\headers{Integration for magnetic wave packet dynamics}{Sebastian Merk, Caroline Lasser}

\title{Structure-preserving integration for magnetic Gaussian wave packet dynamics\thanks{\funding{Funded by the Deutsche Forschungsgemeinschaft (DFG, German Research Foundation) -- TRR~352 -- Project-ID 470903074.}}}

\author{Sebastian Merk\footnotemark[2] \and Caroline Lasser\footnotemark[2]\thanks{Department of Mathematics, Technical University of Munich, Germany 
  (\email{s.merk@tum.de}, \email{classer@tum.de}).}}

\usepackage{amsopn}

\DeclareMathOperator{\qB}{\mathbf{q}}
\DeclareMathOperator{\pB}{\mathbf{p}}
\DeclareMathOperator{\AB}{\mathbf{A}}
\DeclareMathOperator{\VB}{\mathbf{V}}
\DeclareMathOperator{\hB}{\mathbf{h}}
\DeclareMathOperator{\vB}{\mathbf{v}}
\DeclareMathOperator{\BB}{\mathbf{B}}
\DeclareMathOperator{\EB}{\mathbf{E}}
\DeclareMathOperator{\zB}{\mathbf{z}}
\DeclareMathOperator{\stepsize}{\tau}
\DeclareMathOperator{\scs}{\sqrt{\frac{\varepsilon}{2}}}

\DeclareMathOperator{\Yeps}{Y^{(\varepsilon)}}
\DeclareMathOperator{\Id}{\mathrm{Id}}